\documentclass[reqno]{amsart}
\usepackage{amssymb,latexsym,extarrows}
\usepackage{amsfonts,mathrsfs}
\usepackage[margin=1.4in]{geometry}
\usepackage{amsmath}
\usepackage{enumerate}
\usepackage{mathtools}
\usepackage{tikz}
\usetikzlibrary{arrows.meta}
\allowdisplaybreaks

\newtheorem{thm}{Theorem}[section]
\newtheorem{lemma}[thm]{Lemma}
\newtheorem{proposition}[thm]{Proposition}
\newtheorem{corollary}[thm]{Corollary}
\newtheorem{definition}[thm]{Definition}

\theoremstyle{remark}
\newtheorem{remark}{Remark}[section]

\newtheorem*{ack}{Acknowledgement}

\begin{document}
\bibliographystyle{abbrv}

\title[Mass of asymptotically hyperbolic spaces]{Mass of $C^0$-asymptotically hyperbolic spaces  via the normalized Ricci-DeTurck flow}
\author{Yuqiao Li}
\address{Department of Mathematics, Hefei University of Technology, Hefei, 230009, P.R.China.} 
\email{lyq112@mail.ustc.edu.cn}
\thanks{This paper is supported by Tian Yuan Mathematical Foundation of No. 12526577.}
\maketitle

\numberwithin{equation}{section}

\begin{abstract}
    We define a mass function on asymptotically hyperbolic manifolds with continuous metrics via the normalized Ricci–DeTurck flow. 
   This definition coincides with the classical mass for $C^2$ metrics. We also introduce the scalar curvature lower bound for continuous metrics -- a key component in establishing the well-definedness of the $C^0$ mass.
\end{abstract}

\vspace{.2in}

\section{Introduction}
Mass is a crucial quantity on asymptotically hyperbolic manifolds in general relativity. Wang defined the mass and established the positive mass theorem for conformally compact manifolds that are asymptotically  hyperbolic  in a specific sense \cite{Wang01}.
Chrusciel and Herzlich later extended  the mass definition to  more general asymptotically hyperbolic manifolds \cite{Chrusciel03}.
Both Wang's proof and that of Chrusciel-Herzlich for the positive mass theorem on asymptotically hyperbolic manifolds rely on spin structures. 
In three dimensions, Sakovich provided a non-spin proof of the positive mass theorem in the asymptotically hyperbolic setting using the Jang equation  \cite{jang21}.

In this paper, we focus on the mass of asymptotically hyperbolic manifolds with non-smooth metrics.
For a three-dimensional asymptotically hyperbolic initial data set with $W^{1, 2}$ metric, the positive energy theorem was established by Wang and Zhang under nonnegative distributional curvature  \cite{wangzh19}.
 Gicquaud and Sakovich  recently defined the mass aspect function for weakly regular asymptotically hyperbolic manifolds with $W^{1, 2}_{loc}\bigcap L^{\infty}$ metrics by means of cut-off functions   \cite{gicquaud2025}.
Analogously, Lundgren and Meco adopted the same approach to define the ADM mass for asymptotically Euclidean manifolds with $W^{1, 2}_{loc}\bigcap L^{\infty}$ metrics  \cite{lundgren2025}.

On the other hand, Burkhardt-Guim used the Ricci-DeTurck flow to define the weak ADM mass for $C^0$ asymptotically flat metrics, showing that it coincides with the standard ADM mass whenever the latter exists  \cite{MR4685089}.
We thus investigate the mass function on asymptotically hyperbolic manifolds with $C^0$ metrics via the normalized Ricci-DeTurck flow in this work.
Bahuaud  showed that the normalized Ricci flow of conformally compact metrics exists for a short time interval and preserves conformally compactness \cite{Bahu11}. 
Qing-Shi-Wu  established the long-time existence and convergence of the normalized Ricci flow on conformally compact asymptotically hyperbolic manifolds \cite{QingShiWu13}.
Subsequently, Balehowsky and Woolgar proved that the mass of an asymptotically hyperbolic manifold decays exponentially to zero under the normalized Ricci flow \cite{RFofah12}.

Firstly, we define a quantity called $C^0$ local mass function by means of a cutoff function. Particularly, the limit of this local mass is the classical mass defined by Chrusciel-Herzlich when the metric is $C^2$ asymptotically hyperbolic. 
Let $(\mathbb{H}^n, b)$ be the standard hyperbolic manifold.

\begin{definition}\label{mac0}
   Let $(M,g)$ be a $C^0_{\tau}$-asymptotically hyperbolic manifold with a chart at infinity $\Phi:M\backslash K\rightarrow \mathbb{H}^n\backslash K'$ for some compact subsets $K\subset M$, $K'\subset\mathbb{H}^n$ and $\tau>0$ as in Definition \ref{def2.2}.
Let $r>0$ such that the annulus $A(0, 0.9r, 1.1r)=B(1.1r)\backslash\bar{B}(0.9r)\subset \mathbb{H}^n\backslash K'$.
Let $\varphi: \mathbb{R}\rightarrow\mathbb{R}$ be a smooth function with $\int_{0.9}^{1.1}\varphi(l)dl\not=0$. 
Writing $\Phi_*g=g$ and $e=\Phi_*g-b$, we define the $C^0$ local mass function of $g$ with respect to $\varphi$ and $\Phi$ at $r$ by 
\begin{align*}
   & M_{C^0}(g, \Phi, \varphi, r)\\
    :=& \frac{1}{r\int_{0.9}^{1.1}\varphi(l)dl}\bigg[ \int_{\partial A(0, 0.9r, 1.1r)}\cosh s\varphi\left(\frac{s}{r}\right)(e_{ij}\nu^i\nu^j-b^{ij}e_{ij})d\mu_b\\
   & + \int_{A(0, 0.9r, 1.1r)}\left[ \cosh s\varphi'\left( \frac{s}{r} \right)r^{-1}+\varphi\left( \frac{s}{r} \right)\left(n\sinh s +\frac{n-2}{\sinh s}\right) \right]b^{ij}e_{ij}d\mu_b\\
   & + \int_{A(0, 0.9r, 1.1r)}\left[ \varphi\left(\frac{s}{r}\right)\left(\frac{1}{\sinh s}-\sinh s\right)-\cosh s\varphi'\left( \frac{s}{r} \right)r^{-1} \right]\frac{x^ix^j}{\sinh^2 s}e_{ij}d\mu_b \bigg],
\end{align*}
where $x^i$ denotes the coordinate component and $s=|x|_b$ is the geodesic distance.
\end{definition}

We obtain this $C^0$ local mass function by integration by parts which can be seen  in details in Section 2.
Since the Ricci-DeTurck flow exists for a short time with continuous initial metric which is close to the hyperbolic metric due to Simon's theorem \cite{MSimon}, we can define the scalar curvature lower bound via the normalized Ricci-DeTurck flow.

\begin{definition}
    Let $M^n$ be a smooth manifold and $g$ be a $C^0$ Riemannian metric on $M$. For some $\beta\in (0, \frac{1}{2})$, we say that $g$ has scalar curvature bounded below by $\kappa\in \mathbb{R}$ in the $\beta$-weak sense at $x\in M$ with respect to $(\Phi, g_0, g_t)$ if there exists a diffeomorphism $\Phi: U_x\to \Phi(U_x)$ for a neighborhood $U_x$ of $x$ and $\Phi(U_x)\subset\mathbb{H}^n$, and there exists a $C^0$ metric $g_0$ on $\mathbb{H}^n$ and a normalized $b$-flow $(g_t)_{t\in (0, T]}$ for $g_0$ satisfying Corollary \ref{cor3.3} such that
    \begin{align*}
        g_0|_{\Phi(U_x)}=\Phi_*g,\\
        \inf_{C>0}\left(\liminf_{t\to0}\left(\inf_{B(\Phi(x), Ct^{\beta})}R(g_t)\right) \right)\geq\kappa.
    \end{align*}
\end{definition}

We denote that $g$ has scalar curvature bounded below by $\kappa$ in the $\beta$-weak sense at $x$ with respect to $(\Phi, g_0, g_t)$ as $R_{C^0_{\beta}}(g)(x)\geq\kappa$.


Finally, we find an inequality of the $C^0$ local mass with different cutoff functions, which implies that the limit of the $C^0$ local mass exists, and either finite or $+\infty$.

\begin{thm}\label{thm1.4}
Let $(M^n, g)$ be a $C^0_{\tau}$-asymptotically hyperbolic manifold with a fixed chart at infinity $\Phi$ and $\tau>\frac{119n}{128}$, $n\geq3$.
    Let $g_0$ be a continuous metric on $\mathbb{H}^n$ such that $g_0=\Phi_*g$ on $A(0, 0.8r, 12r)$ for $r>r_0$ sufficiently large and 
    \[ ||g_0-b||_{L^{\infty}(\mathbb{H}^n)}<\epsilon, \]
    for some $\epsilon<1$.
    Let $g_t=g(t)$ be the normalized $b$-flow as Theorem \ref{thm3.7}.
    Choose $\eta$ as in Theorem \ref{thm5.7}.
    
     Suppose $R_{C^0_{\beta}}(g)(x)\geq-n(n-1)$, for $\beta\in (0, \frac{1}{2})$ and for any $\Phi(x)\in A(0, 0.8r, 12r)$. 
     Let $\varphi$ be any   smooth positive function with  $supp\varphi\subset\subset(0.9, 1.1)$ and with nonzero integrals over $(0.9, 1.1)$ and $\varphi_{\theta}(l, t)$ be defined as \eqref{varphithe} for any $\theta<T_*$.
     Then the limit 
     \[ \lim_{r\to\infty} M_{C^0}(g, \Phi, \varphi_{e^{-r\eta}}(\cdot, 0), r)\]
exists, and either finite or $+\infty$. Furthermore, 
the limit $\lim_{r\to\infty} M_{C^0}(g, \Phi, \varphi_{e^{-r\eta}}(\cdot, t), r)$ is finite if and only if the last condition in  Theorem 2.9 of \cite{MR4685089} is satisfied.
\end{thm}

It should be noted that our approach is similar to that of Burkhardt-Guim \cite{MR4685089}, but there are essential difficulties specific to asymptotically hyperbolic manifolds. The first challenge we encounter is the need to use the normalized Ricci-DeTurck flow instead of the standard Ricci-DeTurck flow, as the normalized version preserves  hyperbolicity. Since the normalized Ricci-DeTurck flow is not scaling invariant,  proving the almost monotonicity of the $C^0$ local mass in section 4 becomes more intricate.

On the other hand, our assumption is stronger than $\tau>\frac{n}{2}$. This is because some information is lost after  mollification via the normalized Ricci-DeTurck flow, as detailed in Section 5.
In contrast to the results by Chrusciel-Herzlich \cite{Chrusciel03} and Gicquaud-Sakovich \cite{gicquaud2025}, we cannot show that the limit in Definition \ref{mac0} is independent of the choice of the coordinate chart at infinity $\Phi$ and we can only define the mass function for $V^0$ because of Lemma \ref{c2scalar}.

The rest of this paper is organized as follows. In Section 2, we define the $C^0$ local mass function and the $C^2$ local mass function, and clarify their relationship.
In Section 3, we introduce the scalar curvature lower bounds in the $\beta$-weak sense for $C^0$ metrics and establish the almost preservation of this lower bound under the normalized Ricci-DeTurck flow initiated from a continuous metric.
Then we prove that the $C^0$-asymptotically hyperbolicity is almost preserved.
In Section 4, we demonstrate that the definition of the $C^0$ local mass function is almost independent of the choice of the cutoff function and eventually show that the limit of the $C^0$ local mass exists.

\begin{ack}
    The author would like to thank Professor Lamm for his suggestions. We also want to thank Professor Jiayu Li, Professor Binglong Chen and Professor Xiao Zhang for their insightful feedback.
\end{ack}

\vspace{.2in}

\section{$C^0$ and $C^2$ local mass function}

\subsection{$C^0$ local mass function}

Let $(\mathbb{H}^n, b)$ denote the hyperbolic space of dimension $n\geq3$:
\[ \mathbb{H}^n=B_1(0),\quad\quad b=\rho^{-2}\delta,\quad\quad \rho(x)=\frac{1-|x|_{\delta}^2}{2}, \]
where $B_1(0)$ is the open unit ball in $\mathbb{R}^n$, $\delta$ denotes the Euclidean metric, and $x$ is the coordinate in $B_1(0)$.
In \cite{gicquaud2025}, Gicquaud-Sakovich defined the mass aspect function for metrics of $L^{\infty}\bigcap W^{1, 2}$ regularity and distributional scalar curvature on asymptotically hyperbolic manifolds, analogous to the asymptotically flat case introduced in \cite{LeePh}.

Throughout this paper, the Levi-Civita connection associated with the hyperbolic metric $b$ will be denoted by $D$, while the Levi-Civita connection with metric $g$ under consideration will be denoted by $\nabla$. 
Moreover, norms and volume forms with a subscript $b$ or $\delta$ are computed with respect to the hyperbolic metric or Euclidean metric, respectively.

\begin{definition}[Definition 2.1 in \cite{gicquaud2025}]
    Let $(M, g)$ be a complete Riemannian $n$-manifold. Then $(M, g)$ is said to be $W^{1,2}_{\tau}$-asymptotically hyperbolic for $\tau>0$ if there exist compact sets $K\subset M$ and $K'\subset\mathbb{H}^n$, a diffeomorphism $\Phi: M\backslash K\rightarrow\mathbb{H}^n\backslash K'$ and a constant $C>1$ such that $C^{-1}b\leq \Phi_*g\leq Cb$ and $e:=\Phi_*g-b$ satisfies
    \[ \int_{\mathbb{H}^n\backslash K'}\rho^{-2\tau}(|De|_b^2+|e|_b^2)d\mu_b<\infty, \]
    where $\Phi$ is called a chart at infinity for $(M, g)$.
\end{definition}

Now we consider the mass aspect function on asymptotically hyperbolic manifolds with $C^0$ metrics. 

\begin{definition}\label{def2.2}
    Let $M$ be a complete smooth manifold and $g$ be a $C^0$ metric on $M$. Then $(M, g)$ is said to be $C^0_{\tau}$ -asymptotically hyperbolic for $\tau>0$ if there exist compact subsets $K\subset M$ and $K'\subset \mathbb{H}^n$, a diffeomorphism $\Phi: M\backslash K\rightarrow\mathbb{H}^n\backslash K'$ and  $e:=\Phi_*g-b$ satisfies
    \[ |e|_b(x)=O(\rho(x)^{\tau}) \]
    for $x\in \mathbb{H}^n\backslash K'$.
\end{definition}

Let $\mathcal{N}:=\{ V\in C^{\infty}(\mathbb{H}^n)|Hess^bV=Vb \}$. This is an $(n+1)$-dimensional vector space with the basis given by 
\[ V^0=\frac{1+|x|_{\delta}^2}{1-|x|_{\delta}^2}=\frac{1}{\rho}-1, \quad\quad V^i=\frac{x^i}{\rho}, \quad i=1, 2,\cdots, n, \]
where $x^1, \cdots, x^n$ are the coordinate functions on $\mathbb{R}^n$. 
There is a natural correspondence between functions in $\mathcal{N}$ and isometries of Minkowski spacetime preserving the geometry of  the hyperboloid \cite{jang21}.

\begin{definition}[Proposition 4.1 in \cite{gicquaud2025}]
    If $(M, g)$ is $W^{1, 2}_{\tau}$-asymptotically hyperbolic, $e\in C^1_{\tau}(\mathbb{H}^n)$ and $R_g+n(n-1)\in L^1_1(M)$, where $R_g$ is the distributional scalar curvature of $g$ and $\tau>\frac{n}{2}$, then, for any $V\in \mathcal{N}$, the mass aspect function is defined by
\[ p(e, V)=\lim_{r\rightarrow\infty}\int_{S_r(0)}[V(div_be-dtr_be)+tr_bedV-e(DV, \cdot)]\nu_rd\mu_b, \]
where $S_r(0)$ denotes the hyperbolic geodesic sphere in $\mathbb{H}^n$ and $\nu_r$ is the outward unit normal of $S_r(0)$ in $\mathbb{H}^n$.
In particular, we will refer to  $p(e, V^0)$ as the mass function.
\end{definition} 

By the coordinates given by $\Phi$, we can write the mass aspect function as
\begin{equation}\label{ma}
    \begin{split}
        p(e, V)=&\lim_{r\rightarrow\infty}\int_{S_r(0)}[V(D_ke_{ij}b^{jk}\nu_r^i-b^{ij}D_ke_{ij}\nu_r^k)+b^{ij}e_{ij}D_kV\nu_r^k-e_{ij}D^iV\nu^j_r]d\mu_b\\
        =&\lim_{r\rightarrow\infty}\int_{S_r(0)}[V(b^{jk}\nu_r^i-b^{ij}\nu_r^k)D_ke_{ij}+(b^{ij}\nu_r^k-b^{jk}\nu_r^i)e_{ij}D_kV]d\mu_b,
    \end{split}
\end{equation}
which coincides with  Chrusciel-Herzlich's definition \cite{Chrusciel03}.
We use $S(l)$ to represent the hyperbolic geodesic sphere of radius $l$ and $B(l)$ to denote the hyperbolic open ball of radius $l$ centered at the origin. 
We see from \eqref{ma} that the first part of the integrand requires the first derivative of $e$.
Therefore, to define the mass for $e\in C^0_{\tau}$, we need a smooth test function.
The definition of the mass of $e\in C^0_{\tau}$ is based on the following fact.

Let $(M,g)$ be a $C^1_{\tau}$-asymptotically hyperbolic manifold with chart at infinity $\Phi:M\backslash K\rightarrow \mathbb{H}^n\backslash K'$ for some compact subsets $K\subset M$ and $K'\subset\mathbb{H}^n$.
Let $r>0$ such that the annulus $A(0, 0.9r, 1.1r)=B(1.1r)\backslash\bar{B}(0.9r)\subset \mathbb{H}^n\backslash K'$.
If $\varphi: \mathbb{R}\rightarrow\mathbb{R}$ is a smooth function with $\int_{0.9}^{1.1}\varphi(l)dl\not=0$, then we have
    \begin{align}
  &\hspace{1.3em}\frac{\int_{0.9r}^{1.1r}\varphi(\frac{l}{r})\int_{S(l)}[V(b^{jk}\nu^i-b^{ij}\nu^k)D_ke_{ij}+(b^{ij}\nu^k-b^{jk}\nu^i)e_{ij}D_kV]d\mu_bdl}{r\int_{0.9}^{1.1}\varphi(l)dl}\label{lef}\\
        &= \frac{\int_{0.9}^{1.1}\varphi(s)\int_{S(rs)}[V(b^{jk}\nu^i-b^{ij}\nu^k)D_ke_{ij}+(b^{ij}\nu^k-b^{jk}\nu^i)e_{ij}D_kV]d\mu_bds}{\int_{0.9}^{1.1}\varphi(s)ds}\notag\\
        &\xlongrightarrow{r\to\infty} \lim_{r\rightarrow\infty}\int_{S(r)}[V(b^{jk}\nu^i-b^{ij}\nu^k)D_ke_{ij}+(b^{ij}\nu^k-b^{jk}\nu^i)e_{ij}D_kV]d\mu_b.\notag
    \end{align}
We can calculate the numerator of \eqref{lef} with $V=V^0$ by integration by parts.
 By definition, the unit normal vector of $S(s)$ is $\nu_s=\frac{x}{\sinh{s}}$ and $s=\ln\left(\frac{1+|x|_{\delta}}{1-|x|_{\delta}}\right)$, $\sinh{s}=\frac{|x|_{\delta}}{\rho}=\frac{2|x|_{\delta}}{1-|x|_{\delta}^2}$. Then we compute that
\begin{align*}
    & \int_{0.9r}^{1.1r}\varphi\left(\frac{l}{r}\right)\int_{S(l)}V^0(b^{jk}\nu^i-b^{ij}\nu^k)D_ke_{ij}d\mu_bdl\\
    =& \int_{A(0, 0.9r, 1.1r)}V^0\varphi\left(\frac{s}{r}\right)\left(b^{jk}\frac{x^i}{\sinh s}-b^{ij}\frac{x^k}{\sinh s}\right)D_ke_{ij}d\mu_b(x)\\
    =& \int_{\partial A(0, 0.9r, 1.1r)}V^0\varphi\left(\frac{s}{r}\right)\left(b^{jk}\frac{x^i}{\sinh s}-b^{ij}\frac{x^k}{\sinh s}\right)e_{ij}\nu_kd\mu_b\\
   & -\int_{A(0, 0.9r, 1.1r)}D_k\left[V^0\varphi\left(\frac{s}{r}\right)\left(b^{jk}\frac{x^i}{\sinh s}-b^{ij}\frac{x^k}{\sinh s}\right) \right]e_{ij}d\mu_b(x)\\
    =& I+II,
\end{align*}
where
\begin{align*}
    I=&\int_{\partial A(0, 0.9r, 1.1r)}V^0\varphi\left(\frac{s}{r}\right)\left(b^{jk}\frac{x^i}{\sinh s}-b^{ij}\frac{x^k}{\sinh s}\right)e_{ij}\nu_kd\mu_b\\
    =& \int_{\partial A(0, 0.9r, 1.1r)}V^0\varphi\left(\frac{s}{r}\right)(e_{ij}\nu^i\nu^j-b^{ij}e_{ij})d\mu_b,
\end{align*}
and 
\begin{align*}
    II=&-\int_{A(0, 0.9r, 1.1r)}D_k\left[V^0\varphi\left(\frac{s}{r}\right)\left(b^{jk}\frac{x^i}{\sinh s}-b^{ij}\frac{x^k}{\sinh s}\right) \right]e_{ij}d\mu_b(x)\\
    =& -\int_{A(0, 0.9r, 1.1r)} D_kV^0\varphi\left(\frac{s}{r}\right)\left(b^{jk}\frac{x^i}{\sinh s}-b^{ij}\frac{x^k}{\sinh s}\right) e_{ij}d\mu_b(x)\\
    & -\int_{A(0, 0.9r, 1.1r)} V^0D_k\varphi\left(\frac{s}{r}\right)\left(b^{jk}\frac{x^i}{\sinh s}-b^{ij}\frac{x^k}{\sinh s}\right) e_{ij}d\mu_b(x)\\
    &-\int_{A(0, 0.9r, 1.1r)} V^0\varphi\left(\frac{s}{r}\right)D_k\left(b^{jk}\frac{x^i}{\sinh s}-b^{ij}\frac{x^k}{\sinh s}\right) e_{ij}d\mu_b(x).
\end{align*}
By computing directly, there holds
\[ D_ks= \frac{x^k}{\rho|x|_{\delta}} =\frac{x^k}{\rho^2\sinh s}=\nu_k,\]
which implies that
\[ D_k\varphi\left( \frac{s}{r} \right)=\varphi'\left( \frac{s}{r} \right)\frac{x^k}{r\rho^2\sinh{s}}. \]
Then 
\begin{align*}
    & -\int_{A(0, 0.9r, 1.1r)} V^0D_k\varphi\left(\frac{s}{r}\right)\left(b^{jk}\frac{x^i}{\sinh s}-b^{ij}\frac{x^k}{\sinh s}\right) e_{ij}d\mu_b(x)\\
    =& -\int_{A(0, 0.9r, 1.1r)} V^0\varphi'\left( \frac{s}{r} \right)r^{-1}\left(\frac{x^ix^j}{\sinh^2 s} -b^{ij}\right)e_{ij}d\mu_b(x).
\end{align*}
Since $b=\rho^{-2}\delta$, the Christoffel symbol of $(\mathbb{H}^n, b)$ is 
\[ \Gamma_{kj}^l=\rho^{-1}(x^j\delta_{kl}+x^k\delta_{jl}-x^l\delta_{kj}). \]
We have that
\begin{align*}
&D_k(x^j\partial_j)=\partial_k+x^jD_k\partial_j=\partial_k+x^j\Gamma_{kj}^l\partial_l \\
=& \partial_k+x^j\rho^{-1}(x^j\delta_{kl}+x^k\delta_{jl}-x^l\delta_{kj})\partial_l\\
=& \partial_k+\rho^{-1}|x|_{\delta}^2\partial_k=(\rho^{-1}-1)\partial_k.
\end{align*} 
So 
\[ D_kx^i=(\rho^{-1}-1)\delta_{ki}=\cosh{s}\delta_{ki}. \]
Therefore, we get
\begin{align*}
    &D_k\left( b^{jk}\frac{x^i}{\sinh s}-b^{ij}\frac{x^k}{\sinh s} \right)\\
    =& b^{jk}\frac{\cosh s}{\sinh s}\delta_{ik}-nb^{ij}\frac{\cosh s}{\sinh s}-\frac{\cosh s}{\sinh^{2} s}b^{jk}x^i\partial_ks+\frac{\cosh s}{\sinh^2 s}b^{ij}x^k\partial_ks\\
    =& (2-n)\frac{\cosh s}{\sinh s}b^{ij}-\frac{\cosh s}{\sinh^3 s}x^ix^j,
\end{align*}
and
\begin{align*}
   &-\int_{A(0, 0.9r, 1.1r)} V^0\varphi\left(\frac{s}{r}\right)D_k\left(b^{jk}\frac{x^i}{\sinh s}-b^{ij}\frac{x^k}{\sinh s}\right) e_{ij}d\mu_b(x)\\
    =& -\int_{A(0, 0.9r, 1.1r)} V^0\varphi\left(\frac{s}{r}\right)\left[(2-n)\frac{\cosh s}{\sinh s}b^{ij}-\frac{\cosh s}{\sinh^3 s}x^ix^j\right]e_{ij}d\mu_b(x)\\
    =& -\int_{A(0, 0.9r, 1.1r)}\frac{\cosh^2 s}{\sinh s}\varphi\left(\frac{s}{r}\right)\left[(2-n)b^{ij}-\frac{x^ix^j}{\sinh^2 s} \right]e_{ij}d\mu_b(x)
    \end{align*}
By the fact that 
\[ D_kV^0=\partial_k(\rho^{-1}-1)=\rho^{-2}x^k, \]
we have
\begin{align*}
    &-\int_{A(0, 0.9r, 1.1r)} D_kV^0\varphi\left(\frac{s}{r}\right)\left(b^{jk}\frac{x^i}{\sinh s}-b^{ij}\frac{x^k}{\sinh s}\right) e_{ij}d\mu_b(x)\\
    =& -\int_{A(0, 0.9r, 1.1r)}\varphi\left(\frac{s}{r}\right)\left(\frac{x^ix^j}{\sinh s}-b^{ij}\sinh s\right)e_{ij}d\mu_b(x).
\end{align*}

Hence, we obtain
\begin{align*}
   & \int_{0.9r}^{1.1r}\varphi\left(\frac{l}{r}\right)\int_{S(l)}[V^0(b^{jk}\nu^i-b^{ij}\nu^k)D_ke_{ij}+(b^{ij}\nu^k-b^{jk}\nu^i)e_{ij}D_kV^0]d\mu_bdl\\
   = & \int_{A(0, 0.9r, 1.1r)}V^0\varphi\left(\frac{s}{r}\right)\left(b^{jk}\frac{x^i}{\sinh s}-b^{ij}\frac{x^k}{\sinh s}\right)D_ke_{ij}d\mu_b(x)\\
   &+\int_{A(0, 0.9r, 1.1r)}\varphi\left(\frac{s}{r}\right)\left( b^{ij}\frac{x^k}{\sinh s}-b^{jk}\frac{x^i}{\sinh s} \right)e_{ij}D_kV^0 d\mu_b(x)\\
   = & \int_{\partial A(0, 0.9r, 1.1r)}V^0\varphi\left(\frac{s}{r}\right)(e_{ij}\nu^i\nu^j-b^{ij}e_{ij})d\mu_b\\
   & -\int_{A(0, 0.9r, 1.1r)} V^0\varphi'\left( \frac{s}{r} \right)r^{-1}\left(\frac{x^ix^j}{\sinh^2 s} -b^{ij}\right)e_{ij}d\mu_b(x)\\
   & -\int_{A(0, 0.9r, 1.1r)}\frac{\cosh^2 s}{\sinh s}\varphi\left(\frac{s}{r}\right)\left[(2-n)b^{ij}-\frac{x^ix^j}{\sinh^2 s} \right]e_{ij}d\mu_b(x)\\
   &-2\int_{A(0, 0.9r, 1.1r)}\varphi\left(\frac{s}{r}\right)\left(\frac{x^ix^j}{\sinh s}-b^{ij}\sinh s\right)e_{ij}d\mu_b(x)\\
   = & \int_{\partial A(0, 0.9r, 1.1r)}\cosh s\varphi\left(\frac{s}{r}\right)(e_{ij}\nu^i\nu^j-b^{ij}e_{ij})d\mu_b\\ 
   &+ \int_{A(0, 0.9r, 1.1r)}\left[ \cosh s\varphi'\left( \frac{s}{r} \right)r^{-1}+\varphi\left( \frac{s}{r} \right)\left(n\sinh s +\frac{n-2}{\sinh s}\right) \right]b^{ij}e_{ij}d\mu_b\\
   & + \int_{A(0, 0.9r, 1.1r)}\left[ \varphi\left(\frac{s}{r}\right)\left(\frac{1}{\sinh s}-\sinh s\right)-\cosh s\varphi'\left( \frac{s}{r} \right)r^{-1} \right]\frac{x^ix^j}{\sinh^2 s}e_{ij}d\mu_b.
\end{align*}

We now arrive at the following definition of $C^0$ local mass function.

\begin{definition}\label{defc0}
   Let $(M,g)$ be a $C^0_{\tau}$-asymptotically hyperbolic manifold with a chart at infinity $\Phi:M\backslash K\rightarrow \mathbb{H}^n\backslash K'$ for some compact subsets $K\subset M$ and $K'\subset\mathbb{H}^n$ and $\tau>0$.
Let $r>0$ such that the annulus $A(0, 0.9r, 1.1r)=B(1.1r)\backslash\bar{B}(0.9r)\subset \mathbb{H}^n\backslash K'$.
Let $\varphi: \mathbb{R}\rightarrow\mathbb{R}$ be a smooth function with $\int_{0.9}^{1.1}\varphi(l)dl\not=0$. 
Writing $\Phi_*g=g$ and $e=\Phi_*g-b$, we define the $C^0$ local mass function of $g$ with respect to $\varphi$ and $\Phi$ at $r$ by 
\begin{align*}
   & M_{C^0}(g, \Phi, \varphi, r)\\
    :=& \frac{1}{r\int_{0.9}^{1.1}\varphi(l)dl}\bigg[ \int_{\partial A(0, 0.9r, 1.1r)}\cosh s\varphi\left(\frac{s}{r}\right)(e_{ij}\nu^i\nu^j-b^{ij}e_{ij})d\mu_b\\
   & + \int_{A(0, 0.9r, 1.1r)}\left[ \cosh s\varphi'\left( \frac{s}{r} \right)r^{-1}+\varphi\left( \frac{s}{r} \right)\left(n\sinh s +\frac{n-2}{\sinh s}\right) \right]b^{ij}e_{ij}d\mu_b\\
   & + \int_{A(0, 0.9r, 1.1r)}\left[ \varphi\left(\frac{s}{r}\right)\left(\frac{1}{\sinh s}-\sinh s\right)-\cosh s\varphi'\left( \frac{s}{r} \right)r^{-1} \right]\frac{x^ix^j}{\sinh^2 s}e_{ij}d\mu_b \bigg],
\end{align*}
where $x^i$ denotes the coordinate component and $s=|x|_b$ is the geodesic distance.
\end{definition}

In particular, if $e\in C^1_{\tau}(\mathbb{H}^n)$ and the mass aspect function $p(e,V)$ exists, by the above process, \eqref{lef} implies that 
\[ p(e, V^0, \Phi)=\lim_{r\rightarrow\infty}M_{C_0}(g, \Phi, \varphi, r), \]
which means that the $C^0$ local mass is a generalization of the mass  function.

\subsection{$C^2$ local mass aspect function}

\begin{definition}\label{c2mass}
    If $g$ is a $C^2$ Riemannian metric on $M$ and $\Phi$ is a coordinate chart of $M$ such that $\Phi_*g$ is a $C^2$ metric on a region of $\mathbb{H}^n$ containing $S(r)$ for some $r>0$, then we define the $C^2$ local mass aspect function as
    \[ M_{C^2}(g, \Phi, r, V):=\int_{S(r)}[V(b^{jk}\nu_r^i-b^{ij}\nu_r^k)D_ke_{ij}+(b^{ij}\nu_r^k-b^{jk}\nu_r^i)e_{ij}D_kV]d\mu_b, \]
    where we write $e=\Phi_*g-b$ and $V\in \mathcal{N}$. 
\end{definition}

We will omit $\Phi$ for a fixed coordinate chart in $M_{C^2}(g, \Phi, r, V)$.
In particular, if the mass aspect function $p(e, V)$ exists, then there hold
    \[p(e, V, \Phi)=\lim_{r\to \infty}M_{C^2}(g, \Phi, r, V)\]
    and
    \[ M_{C^0}(g, \Phi, \varphi, r)=\frac{\int_{0.9}^{1.1}\varphi(l)M_{C^2}(g, \Phi, rl, V^0)dl}{\int_{0.9}^{1.1}\varphi(l)dl}. \]

\begin{lemma}\label{c2scalar}
    Suppose that $g$ is a $C^2$ asymptotically hyperbolic Riemannian metric on $\mathbb{H}^n\backslash B(r_0)$. Let $e=g-b$. For $r_0<r_1<r_2$ and $V\in \mathcal{N}$, we have
    \[ M_{C^2}(g, r_2, V)-M_{C^2}(g, r_1, V)=\int_{A(0, r_1, r_2)}[V(R_g+n(n-1))+Q]d\mu_b, \]
    where $Q=Q(V, dV, e, De, D^2e)$ satisfies $|Q|\lesssim(|V|+|dV|)(|De|^2+|e|^2)+|V||e||D^2e|$ and norms are computed by $b$.
\end{lemma}

\begin{proof}
    By proposition  2.2 in \cite{gicquaud2025}, the first order variation of the scalar curvature near the hyperbolic metric $b$ is 
\begin{equation*}
    R_g=-n(n-1)+(n-1)tr_be+D_i[g^{jl}g^{ik}(D_je_{kl}-D_ke_{jl})]+Q(e, De),
\end{equation*}
where $|Q(e, De)|\lesssim |e|^2+|De|^2$. 
We have by integration by parts that
\begin{align*}
    &\int_{A(0, r_1, r_2)}V[R_g+n(n-1)]d\mu_b\\
    =& \int_{A(0, r_1, r_2)}V\left[(n-1)tr_be+D_i(g^{jl}g^{ik}(D_je_{kl}-D_ke_{jl}))+Q(e, De)\right]d\mu_b\\
    =& \int_{A(0, r_1, r_2)}[V(n-1)tr_be+D_i(V(g^{jl}g^{ik}-g^{ij}g^{kl})D_je_{kl})-D_iV(g^{jl}g^{ik}-g^{ij}g^{kl})D_je_{kl}\\
    &\qquad\qquad\quad+VQ(e, De)]d\mu_b\\
    =& \int_{A(0, r_1, r_2)} \big[ V(n-1)tr_be+D_i(V(g^{jl}g^{ik}-g^{ij}g^{kl})D_je_{kl})-D_j(D_iV(g^{jl}g^{ik}-g^{ij}g^{kl})e_{kl})\\
    & \qquad\qquad\quad+D_j(D_iV(g^{ik}g^{jl}-g^{ij}g^{kl}))e_{kl}+VQ(e, De) \big]d\mu_b\\
    = & \int_{A(0, r_1, r_2)}[(n-1)Vtr_be+D_jD_iV(g^{ik}g^{jl}-g^{ij}g^{kl})e_{kl}+D_iVD_j(g^{ik}g^{jl}-g^{ij}g^{kl})e_{kl}\\
    &\qquad\qquad\quad+VQ(e, De)]d\mu_b\\
   & +\int_{S(r_2)}[(g^{ik}g^{jl}-g^{ij}g^{kl})VD_je_{kl}-(g^{ik}g^{jl}-g^{ij}g^{kl})D_jVe_{kl}]\nu_id\mu_b\\
   &-\int_{S(r_1)}[(g^{ik}g^{jl}-g^{ij}g^{kl})VD_je_{kl}-(g^{ik}g^{jl}-g^{ij}g^{kl})D_jVe_{kl}]\nu_id\mu_b.
\end{align*}
Since $|b^{jl}b^{ik}-g^{jl}g^{ik}|\lesssim|e|$, we see that
\begin{align*}
   & \int_{S(r_2)}[(g^{ik}g^{jl}-g^{ij}g^{kl})VD_je_{kl}-(g^{ik}g^{jl}-g^{ij}g^{kl})D_jVe_{kl}]\nu_id\mu_b\\
    =&\int_{S(r_2)}[(b^{ik}b^{jl}-b^{ij}b^{kl})(VD_je_{kl}-D_jVe_{kl})\nu_i+Q_1(V, dV, e, De)\nu_i]d\mu_b\\
    =& \int_{S(r_2)}[(b^{jl}\nu^k-b^{kl}\nu^j)(VD_je_{kl}-D_jVe_{kl})+Q_1(V, dV, e, De)\nu_i]d\mu_b\\
    = & M_{C^2}(g, r_2, V)+\int_{S(r_2)}Q_1(V, dV, e, De)\nu_id\mu_b,
\end{align*}
where $|Q_1(V, dV, e, De)|\lesssim |V||e||De|+|dV||e|^2$.
By the fact that $HessV=Vb$, we have
\begin{align*}
    &\int_{A(0, r_1, r_2)}[(n-1)Vtr_be+D_jD_iV(g^{ik}g^{jl}-g^{ij}g^{kl})e_{kl}+D_iVD_j(g^{ik}g^{jl}-g^{ij}g^{kl})e_{kl}]d\mu_b\\
    =& \int_{A(0, r_1, r_2)}Q_2(e, dV, De)d\mu_b,
\end{align*}
where $|Q_2(e, dV, De)|\lesssim |V||e|^2+|dV||De||e|$.
Hence, we obtain that
\begin{align*}
    &\int_{A(0, r_1, r_2)}V[R_g+n(n-1)]d\mu_b\\
    = & \int_{A(0, r_1, r_2)} [Q_2(e, dV, De)+VQ(e, De)]d\mu_b
    +M_{C^2}(g, r_2, V)
    -M_{C^2}(g, r_1, V)\\
   & +\int_{S(r_2)}Q_1(V, dV, e, De)\nu_id\mu_b-\int_{S(r_1)}Q_1(V, dV, e, De)\nu_id\mu_b\\
   = & \int_{A(0, r_1, r_2)} [Q_2(e, dV, De)+VQ(e, De)]d\mu_b
    +M_{C^2}(g, r_2, V)
    -M_{C^2}(g, r_1, V)\\
    & + \int_{A(0, r_1, r_2)}D_iQ_1d\mu_b.
\end{align*} 
Since
\[ |Q_2(e, dV, De)+VQ(e, De)|\lesssim(|V|+|dV|)(|De|^2+|e|^2), \]
\[ |DQ_1(V, dV, e, De)|\lesssim(|V|+|dV|)(|De|^2+|e|^2)+|V||e||D^2e|,  \]
we arrive at the result.
\end{proof}

\vspace{.2in}

\section{Scalar curvature lower bound for $C^0$ metrics}

\subsection{Ricci-DeTurck flow from $C^0$ metrics}
We now aim to define the scalar curvature lower bound for $C^0$ metrics, which we accomplish by leveraging the Ricci-DeTurck flow.
To regularize continuous metrics, we appeal to a result by Simon \cite{MSimon}.
We begin with the Ricci-DeTurck flow associated with a background metric $h$;
 following \cite{MSimon}, we refer to this as the $h$-flow to highlight its dependence on $h$.

Let $(M, h)$ be a complete smooth Riemannian manifold  such that for every $i\in \mathbb{N}$, there exists $k_i>0$ satisfying
\begin{equation}\label{h}
|\tilde{\nabla}^iRm(h)|\leq k_i, 
\end{equation}
where $\tilde{\nabla}$ denotes the covariant derivative with respect to $h$.
A smooth family of metrics $g(t)$ on $M\times (0, T]$ is said to be a solution to the $h$-flow if it satisfies
\begin{equation}\label{gflow}
    \left\{\begin{array}{l}
\partial_{t} g_{i j}=-2 R_{i j}+\nabla_{i} W_{j}+\nabla_{j} W_{i} , \\
W^{k}=g^{p q}\left(\Gamma_{p q}^{k}(g)-\widetilde{\Gamma}_{p q}^{k}(h)\right) ,
\end{array}\right.
\end{equation}
where $R_{ij}$ denotes the Ricci tensor of $g(t)$, $\Gamma(g)$ and $\widetilde{\Gamma}(h)$ denote the Christoffel symbols of $g$ and $h$, respectively.
If the initial metric $g_0$ is smooth, the Ricci flow is equivalent to the Ricci-DeTurck flow. 
Let $\chi_t$ be the diffeomorphism defined by 
\begin{equation}\label{diff}
    \left\{\begin{aligned}
\frac{\partial}{\partial t} \chi_{t}(x) & =-W\left(\chi_{t}(x), t\right); \\
\chi_{0}(x) & =x.
\end{aligned}\right.
\end{equation}
Then the pullback of the Ricci-DeTurck flow $\hat{g}(t)=\chi_t^*g(t)$ satisfies the Ricci flow
\[\partial_t\hat{g}=-2Ric(\hat{g}),\]
with initial condition $\hat{g}(0)=g(0)=g_0$.
For $\epsilon>0$, a continuous metric $g$ is said to be $(1+\epsilon)$-close to $h$ if 
\[ (1+\epsilon)^{-1}h\leq g\leq (1+\epsilon)h. \]
In \cite{MSimon}, Simon proved the following result for continuous metrics.

\begin{thm}[Theorem 5.2 in \cite{MSimon}]\label{simon}
    There is $\epsilon(n)>0$ such that the following is true:
Let $(M, h)$ be a complete manifold satisfying \eqref{h}. 
If $g_0$ is a continuous metric on $M$ such that $g_0$ is $(1+\epsilon(n))$-close to $h$, then the $h$-flow \eqref{gflow} admits a smooth solution $g(t)$ on $M\times (0, T_0]$ for some $T_0(n, k_0)>0$ such that 
\begin{enumerate}
    \item \[ \lim_{t\to0}\sup_{\Omega}|g(t)-g_0|=0, \quad\forall\Omega\subset\subset M; \]
    \item For all $i\in \mathbb{N}$, there is $C_i>0$ depending only on $n, k_0, \cdots,k_i$ so that
    \begin{equation}\label{gde}
        \sup_{M}|\tilde{\nabla}^ig(t)|\leq \frac{C_i}{t^{i/2}};
    \end{equation}  
    \item $g(t)$ is $(1+2\epsilon(n))$-close to $h$ for all $t\in (0, T_0]$,
\end{enumerate}
where the norm and connection are with respect to $h$.
\end{thm}

\vspace{.2in}

\subsection{Heat kernel estimates under the Ricci-DeTurck flow}

Given a solution $g(t)$ to the $h$-flow \eqref{gflow} on a smooth manifold $M$, let $\hat{K}(x, t; y, s)$ denote the heat kernel associated with the Ricci flow solution $\hat{g}(t)=\chi_t^*g(t)$. That is, for a fixed point $(y, s)\in M\times I$, 
\begin{align*}
   &\partial_t\hat{K}(x, t; y, s)=\Delta_{\hat{g}_t, x}\hat{K}(x, t; y, s), \\
    & \lim_{t\to s}\hat{K}(x, t; y, s)=\delta_y,
\end{align*} 
and for all $s<t$, 
\[ \int_M\hat{K}(x, t; y, s)d\mu_s(y)=1. \]

The push-forward  of $\hat{K}(x, t; y, s)$ under the diffeomorphism $\chi_t$ in \eqref{diff} is
\[ K(x, t; y, s)=\hat{K}(\chi^{-1}_t(x), t; \chi^{-1}_s(y), s), \]
and it serves as the heat kernel for the $h$-flow. It satisfies
\begin{align*}
   &\partial_tK(x, t; y, s)=\Delta_{{g}_t, x}{K}(x, t; y, s)+\langle W, \nabla_xK(x, t; y, s)\rangle, \\
    & \lim_{t\to s}{K}(x, t; y, s)=\delta_y,
\end{align*} 
and for all $s<t$, 
\[ \int_M{K}(x, t; y, s)d\mu_s(y)=1. \]

\begin{lemma}\label{heat}
    Suppose that $g(t)$ is a solution to \eqref{gflow} on $\mathbb{H}^n$ with background metric $h=b$ satisfying $|g(t)-b|_b<\epsilon$ and Theorem \ref{simon}.
    Let $K(x, t; y, s)$ be the heat kernel of the $b$-flow. 
    Then there exist constants $C=C(n, \epsilon)<\infty$, $D=D(n, \epsilon)<\infty$ such that for any $x, y\in \mathbb{H}^n$, $t>0$ sufficiently small and $\frac{t}{2}\leq s<t$,
    \[ K(x, t; y, s)\leq \frac{C}{(t-s)^{n/2}}\exp\left( -\frac{d_b(x,y)^2}{D(t-s)} \right) \]
    and for any $r>0$
    \[ \int_{\mathbb{H}^n\backslash B(x, r)}K(x, t; y, s)d\mu_s(y)<C\exp\left( -\frac{r^2}{D(t-s)} \right), \]
    where $B(x, r)$ is the geodesic ball centered at $x$ with radius $r$ in $(\mathbb{H}^n, b)$.
\end{lemma}

\begin{proof}
Let $\hat{K}(x, t; y, s)$ be the heat kernel for the Ricci flow solution $\hat{g}(t)=\chi_t^*g(t)$.
By Lemma 2.4 of \cite{paula19}, for all $x, y\in \mathbb{H}^n$, $\frac{t}{2}\leq s<t$, there exist constants $C=C(n, \epsilon)<\infty$, $D=D(n)<\infty$, such that
\begin{equation*}
    \hat{K}(x, t; y, s)\leq C\exp\left( -\frac{d^2_{\hat{g}(\frac{t}{2})}(x, y)}{D(t-s)} \right)\text{vol}^{-\frac{1}{2}}B_{\hat{g}(\frac{t}{2})}\left( x, \sqrt{\frac{t-s}{2}} \right)\text{vol}^{-\frac{1}{2}}B_{\hat{g}(\frac{t}{2})}\left( y, \sqrt{\frac{t-s}{2}} \right).
\end{equation*}
Pushing forward $\hat{K}(x, t; y, s)$ by $\chi$, we get
\begin{equation}
\begin{split}
    & K(x, t; y, s)=\hat{K}(\chi^{-1}_t(x), t; \chi^{-1}_s(y), s)\\
    \leq & C\exp\left( -\frac{d^2_{\hat{g}(\frac{t}{2})}(\chi^{-1}_t(x), \chi^{-1}_s(y))}{D(t-s)} \right)\text{vol}^{-\frac{1}{2}}B_{\hat{g}(\frac{t}{2})}\left( \chi^{-1}_t(x), \sqrt{\frac{t-s}{2}} \right)\\
    &\times\text{vol}^{-\frac{1}{2}}B_{\hat{g}(\frac{t}{2})}\left( \chi^{-1}_s(y), \sqrt{\frac{t-s}{2}} \right). 
\end{split}
\end{equation}
By the Jensen inequality, there holds
\[ d^2_{\hat{g}(t)}(\chi^{-1}_t(x), \chi^{-1}_s(y))\geq \frac{1}{2}d^2_{\hat{g}(t)}(\chi^{-1}_t(x), \chi^{-1}_{t}(y))-d^2_{\hat{g}(t)}(\chi^{-1}_t(y), \chi^{-1}_s(y)). \]
Since by \eqref{gde}, we have
\[ |W|(t)\lesssim |\widetilde{\nabla} g|(t)<\frac{C}{\sqrt{t}}, \]
which implies that for any $x\in \mathbb{H}^n$ and any $0<t_1<t_2$,
\[ d_{\hat{g}(t_1)}(\chi_{t_1}(x), \chi_{t_2}(x))\leq C(\sqrt{t_2}-\sqrt{t_1}). \]
By (3) of Theorem \ref{simon}, we see that
\begin{align*}
    &\exp\left( -\frac{d^2_{\hat{g}(\frac{t}{2})}(\chi^{-1}_t(x), \chi^{-1}_s(y))}{D(t-s)} \right)\leq C\exp\left( -\frac{d^2_{\hat{g}(t)}(\chi^{-1}_t(x), \chi^{-1}_s(y))}{D(t-s)} \right)\\
    \leq & C\exp\left( -\frac{d^2_{\hat{g}(t)}(\chi^{-1}_t(x), \chi^{-1}_{t}(y))}{2D(t-s)} \right)\exp\left( \frac{d^2_{\hat{g}(t)}(\chi^{-1}_t(y), \chi^{-1}_{s}(y))}{D(t-s)} \right)\\
    \leq & C\exp\left( -\frac{d^2_{g(t)}(x,y)}{2D(t-s)} \right)\exp\left(\frac{C(\sqrt{t}-\sqrt{s})^2}{D(t-s)}  \right)\\
    \leq & C \exp\left( -\frac{d^2_b(x,y)}{D(t-s)} \right).
\end{align*}

Since $g(t)$ is $(1+2\epsilon)$-close to $b$, the volume can be estimated by
\[ \text{vol} B_{\hat{g}(\frac{t}{2})}\left( \chi^{-1}_t(x), \sqrt{\frac{t-s}{2}} \right)\geq C\text{vol} B\left( x, \sqrt{\frac{t-s}{2}} \right)\geq C(t-s)^{\frac{n}{2}}, \]
where in the last inequality we require $(t-s)$ sufficiently small.
Therefore, we obtain 
\[ K(x, t; y, s)\leq \frac{C}{(t-s)^{n/2}}\exp\left( -\frac{d_b(x,y)^2}{D(t-s)} \right). \]

Integrating the above inequality over $\mathbb{H}^n\backslash B(x, r)$ for any $x\in \mathbb{H}^n$ and any $r>0$ with respect to $y$, we get 
\begin{align*}
    &\int_{\mathbb{H}^n\backslash B(x, r)}K(x, t; y, s)d\mu_s(y)\\
    \leq& C \int_{\mathbb{H}^n\backslash B(x, r)} (t-s)^{-\frac{n}{2}}\exp\left( -\frac{d_b(x,y)^2}{D(t-s)} \right)d\mu_b(y)
\end{align*}

Using hyperbolic polar coordinates centered at $x$:
$$
y=\exp_x(\rho\omega),\qquad
\rho=d_b(x,y),\qquad
\omega\in S^{n-1},
$$
then
$$
\begin{aligned}
&\int_{\mathbb H^n\setminus B_b(x,r)}
K(x,t;y,s)\,d\mu_{g(s)}(y)\\
&\le
C(t-s)^{-n/2}
\int_r^\infty
\exp\!\left(-\frac{\rho^2}{D(t-s)}\right)
\sinh^{n-1}\rho\,d\rho.
\end{aligned}
$$
For $(t-s)$ sufficiently small, we obtain
$$
(t-s)^{-n/2}
\int_r^\infty
\exp\!\left(-\frac{\rho^2}{D(t-s)}\right)
\sinh^{n-1}\rho\,d\rho
\le
C\exp\!\left(-\frac{r^2}{2D(t-s)}\right).
$$
\end{proof}

\vspace{.2in}

\subsection{Normalized Ricci-DeTurck flow}

In contrast to  asymptotically flat manifolds, the hyperbolic metric  expands under the Ricci flow. We therefore resort to the normalized Ricci flow, which is equivalent to the standard Ricci flow up to a time reparametrization.
The normalized Ricci flow consists of a family of metrics $g(t)$ on an $n$-manifold $M$ evolving by
\begin{equation*}
    \partial_tg(t)=-2Ric(g(t))-2(n-1)g(t),
\end{equation*}
with initial condition $g(0)=g_0$.
Via a time-dependent diffeomorphism $\chi_t$ of the form \eqref{diff}, we derive the normalized $h$-flow with respect to a smooth background  metric $h$:
\begin{equation}\label{nrf}
    \partial_tg_{ij}=-2Ric(g(t))_{ij}-2(n-1)g_{ij}+\nabla_i W_j+\nabla_j W_i.
\end{equation}
In what follows, we take the background metric to be $b$ (the hyperbolic metric on $\mathbb{H}^n$) and refer to the flow $g(t)$ in \eqref{nrf} as the normalized $b$-flow.
It is straightforward to verify that 
\[ \bar{g}(t)=(1+2(n-1)t)g(\bar{t})=e^{2(n-1)\bar{t}}g(\bar{t}), \]
where $\bar{t}=\frac{1}{2(n-1)}\ln{(1+2(n-1)t)}$, satisfies the $b$-flow equation \eqref{gflow}.
By Theorem \ref{simon}, we obtain the following corollary for the normalized $b$-flow initiated from a continuous metric.

\begin{corollary}\label{cor3.3}
    If $g_0$ is a continuous metric on $\mathbb{H}^n$ such that $g_0$ is $(1+\epsilon(n))$-close to $b$, then the normalized $b$-flow \eqref{nrf} admits a smooth solution on $\mathbb{H}^n\times(0, T_0]$ for some $T_0(n)>0$ such that
    \begin{enumerate}
    \item \[ \lim_{t\to0}\sup_{\Omega}|e^{2(n-1)t}g(t)-g_0|=0, \quad\forall\Omega\subset\subset \mathbb{H}^n; \]
    \item For all $i\in \mathbb{N}$, there is $C_i>0$ depending only on $n$ so that
    \begin{equation}\label{gde1}
        \sup_{\mathbb{H}^n}|D^ig(t)|\leq \frac{C_i}{({e^{2(n-1)t}-1})^{i/2}};
    \end{equation}  
    \item $e^{2(n-1)t}g(t)$ is $(1+2\epsilon(n))$-close to $b$ for all $t\in (0, T_0]$.
\end{enumerate}
\end{corollary}

Let $K(x, t; y, s)$ be the heat kernel of the $b$-flow $\bar{g}(t)$ as introduced in Section 3.2.    
Direct computation shows that
\[ \Delta_{\bar{g}(t)}=(1+2(n-1)t)^{-1}\Delta_{g(\bar{t})}, \]
\[ \frac{d\bar{t}}{dt}=\frac{1}{1+2(n-1)t}. \]
It therefore follows that
\[ \partial_{\bar{t}}K(x, t; y, s)=\Delta_{g(\bar{t}), x}K(x, t; y, s)+\langle W, \nabla_xK(x, t; y,s )\rangle. \]
Recall that $t=\frac{e^{2(n-1)\bar{t}}-1}{2(n-1)}$. Substitution this relation, we obtain that
\begin{align*}
    \partial_{\bar{t}}K\left(x, \frac{e^{2(n-1)\bar{t}}-1}{2(n-1)}; y, \frac{e^{2(n-1)\bar{s}}-1}{2(n-1)}\right)&=\Delta_{g(\bar{t}),x}K+\langle W, \nabla K\rangle,\\
    \lim_{\bar{t}\to \bar{s}}K\left(x, \frac{e^{2(n-1)\bar{t}}-1}{2(n-1)}; y, \frac{e^{2(n-1)\bar{s}}-1}{2(n-1)}\right)& =\delta_y(x),\\
    \int_{\mathbb{H}^n}K\left(x, \frac{e^{2(n-1)\bar{t}}-1}{2(n-1)}; y, \frac{e^{2(n-1)\bar{s}}-1}{2(n-1)}\right)d\mu_{g_{\bar{s}}}(y)&=e^{-n(n-1)\bar{s}}.
\end{align*}
By Lemma \ref{heat}, we derive the following results.

\begin{corollary}\label{cor3.4}
    Let $g(\bar{t})$ be the normalized $b$-flow on $\mathbb{H}^n\times (0, T_0]$ as stated in Corollary \ref{cor3.3}. Then $K\left(x, \frac{e^{2(n-1)\bar{t}}-1}{2(n-1)}; y, \frac{e^{2(n-1)\bar{s}}-1}{2(n-1)}\right)$ is the unnormalized heat kernel of $g(\bar{t})$. There exist constants $C=C(n)<\infty$ and $D=D(n)<\infty$ such that for any $x, y\in\mathbb{H}^n$, $\bar{t}>0$ sufficiently small and $\frac{e^{2(n-1)\bar{t}}-1}{2}\leq e^{2(n-1)\bar{s}}-1<e^{2(n-1)\bar{t}}-1$, the following estimate holds for all $r>0$:
    \begin{equation*}
    \begin{split}
        &\int_{\mathbb{H}^n\backslash B(x, r)}K\left(x, \frac{e^{2(n-1)\bar{t}}-1}{2(n-1)}; y, \frac{e^{2(n-1)\bar{s}}-1}{2(n-1)}\right)d\mu_{g_{\bar{s}}}(y)\\
        <&Ce^{-n(n-1)\bar{s}}\exp\left( -\frac{r^2}{D(e^{2(n-1)\bar{t}}-e^{2(n-1)\bar{s}})}\right).
        \end{split}
    \end{equation*}
   Define 
    \[ H(x,\bar{t}; y, \bar{s})=e^{2(n-1)(\bar{t}-\bar{s})}K\left(x, \frac{e^{2(n-1)\bar{t}}-1}{2(n-1)}; y, \frac{e^{2(n-1)\bar{s}}-1}{2(n-1)}\right). \]
    Then $H$ satisfies
    \begin{align*}
        &\partial_{\bar{t}}H(x,\bar{t}; y, \bar{s})=\Delta_{g(\bar{t})}H(x,\bar{t}; y, \bar{s})+\langle W, \nabla H(x,\bar{t}; y, \bar{s})\rangle+2(n-1)H(x,\bar{t}; y, \bar{s}),\\
        &\int_{\mathbb{H}^n}H(x,\bar{t}; y, \bar{s})d\mu_{g_{\bar{s}}}(y)= e^{2(n-1)(\bar{t}-\bar{s})-n(n-1)\bar{s}},\\
    \end{align*}
    and 
    \begin{align*}
        &\int_{\mathbb{H}^n\backslash B(x, r)}H(x,\bar{t}; y, \bar{s})d\mu_{g_{\bar{s}}}(y)
        <Ce^{2(n-1)(\bar{t}-\bar{s})-n(n-1)\bar{s}}\exp\left( -\frac{r^2}{D(e^{2(n-1)\bar{t}}-e^{2(n-1)\bar{s}})}\right).
    \end{align*}
\end{corollary}

\vspace{.2in}

\subsection{Scalar curvature lower bound of $C^0$ metrics in the $\beta$-weak sense}

\begin{thm}\label{scalart}
    Suppose $g(t)$ is a smooth solution to the normalized $b$-flow on $\mathbb{H}^n$ satisfying Corollary \ref{cor3.3} with $|g_0-b|_b<\epsilon$. For any fixed $x\in\mathbb{H}^n$, $t>0$ sufficiently small, $\beta\in (0, \frac{1}{2})$ and a sequence $\{t_k\}_{k\geq1}$ defined by $t_k=\frac{1}{2(n-1)}\ln{(\frac{e^{2(n-1)t_{k-1}}+1}{2})}<\frac{3}{4}t_{k-1}$ for $k\geq1$ with $t_0=t$, the following estimate holds:
    \begin{align*}
        R(x, t)\geq& \inf_{C>0}(\liminf_{t\to0}(\inf_{y\in B(x, Ct^{\beta})}R(y, t)))e^{2(n-1)\sum_{i=0}^{\infty}t_i}\prod_{i=0}^{\infty}\left( \frac{e^{2(n-1)t_i}+1}{2}\right)^{-(1+n/2)}\\
        & -C\sum_{i=1}^{\infty}\frac{1}{(e^{2(n-1)t}-1)/2^i}\exp\left( -\frac{(e^{2(n-1)t}-1)^{2\beta-1}}{2^{(2\beta-1)i}D} \right),
    \end{align*}  
    where $C, D$ are constants depending only on $n, \epsilon, \beta$ and $R(x, t)$ denotes the scalar curvature of $g(t)$ at $(x, t)$.
\end{thm}

\begin{proof}
Let $H(x, t; y, s)$ be as in Corollary \ref{cor3.4}.
Under the normalized $b$-flow $g(t)$,
 the evolution equation for the scalar curvature (Proposition 2.3.9. of \cite{topping}) is
\begin{equation*}
\partial_{t}R(g(t))=\Delta_{g(t)}R(g(t))+2(n-1)R(g(t))+\langle W, \nabla R(g(t))\rangle+2|Ric(g(t))|^2.
\end{equation*}
For any $x\in \mathbb{H}^n$ and any sufficiently small $0<s<t\leq\frac{1}{2(n-1)}\ln{(\frac{3}{2})}$ , we have 
\[ R(x, t)\geq\int_{\mathbb{H}^n}H(x, t; y, s)R(y, s)d\mu_s(y), \]
    where $d\mu_s(y)=d\mu_{g(s)}(y)$.
    Note that (2) in Corollary \ref{cor3.3} gives the bound
\begin{equation}\label{sca}
    |R(g(t))|\leq \frac{c(n)}{e^{2(n-1)t}-1}.
\end{equation}
Choose $t_1=\frac{1}{2(n-1)}\ln{(\frac{e^{2(n-1)t}+1}{2})}<\frac{3}{4}t$ so that $e^{2(n-1)t_1}-1=\frac{e^{2(n-1)t}-1}{2}$. 
Using Corollary \ref{cor3.4} and \eqref{sca}, we get
\begin{align*}
  &  R(x, t)\\
  \geq &\int_{\mathbb{H}^n}H(x, t; y, t_1)R(y, t_1)d\mu_{t_1}(y)\\
  = & \int_{B(x, r_1)} H(x, t; y, t_1)R(y, t_1)d\mu_{t_1}(y)
  +\int_{\mathbb{H}^n\backslash B(x, r_1)}H(x, t; y, t_1)R(y, t_1)d\mu_{t_1}(y)\\
  \geq & a_1 \int_{B(x, r_1)} H(x, t; y, t_1)d\mu_{t_1}(y)
  -\frac{c(n)}{(e^{2(n-1)t}-1)/2}\int_{\mathbb{H}^n\backslash B(x, r_1)}H(x, t; y, t_1)d\mu_{t_1}(y)\\
  = & a_1e^{2(n-1)t}\left( \frac{e^{2(n-1)t}+1}{2} \right)^{-(1+n/2)}\\
 & -\left(a_1+\frac{c(n)}{(e^{2(n-1)t}-1)/2}\right)\int_{\mathbb{H}^n\backslash B(x, r_1)}H(x, t; y, t_1)d\mu_{t_1}(y)\\
  \geq &a_1e^{2(n-1)t}\left( \frac{e^{2(n-1)t}+1}{2} \right)^{-(1+n/2)}-\frac{2c(n)}{(e^{2(n-1)t}-1)/2}\int_{\mathbb{H}^n\backslash B(x, r_1)}H(x, t; y, t_1)d\mu_{t_1}(y)\\
  \geq & a_1e^{2(n-1)t}\left( \frac{e^{2(n-1)t}+1}{2} \right)^{-(1+n/2)}-\frac{2c(n)}{(e^{2(n-1)t}-1)/2}\exp\left( -\frac{r_1^2}{D(e^{2(n-1)t}-1)/2} \right),
\end{align*}
where $r_1>0$ is to be chosen later, 
\[a_1= \inf_{y\in B(x, r_1)}R(y, t_1)=R(x_1, t_1)\leq \frac{c(n)}{(e^{2(n-1)t}-1)/2}\]
for some $x_1\in B(x, r_1)$, and we have used the fact that 
\[e^{2(n-1)t}\left( \frac{e^{2(n-1)t}+1}{2} \right)^{-(1+n/2)}<1\]
in the last inequality. 

Repeating the same procedure, we choose 
\[t_2=\frac{1}{2(n-1)}\ln{(\frac{e^{2(n-1)t_1}+1}{2})}<\frac{3}{4}t_1,\]
and obtain
\[ a_1\geq a_2e^{2(n-1)t_1}\left( \frac{e^{2(n-1)t_1}+1}{2} \right)^{-(1+n/2)}-\frac{2c(n)}{(e^{2(n-1)t}-1)/2^2}\exp\left( -\frac{r_2^2}{D(e^{2(n-1)t}-1)/2^2} \right), \]
where  $a_2=\inf_{y\in B(x_1, r_2)}R(y, t_2)$ for $r_2>0$ to be chosen.

Proceeding inductively, we obtain:
\begin{enumerate}
    \item A sequence $\{t_k\}$ satisfying 
\begin{enumerate}
    \item $t_0=t$,
    \item $t_k=\frac{1}{2(n-1)}\ln{(\frac{e^{2(n-1)t_{k-1}}+1}{2})}<\frac{3}{4}t_{k-1}$ for $k\geq1$;
\end{enumerate}
\item A series of points $\{x_k\}$ with $x_0=x$ and $x_k\in B(x_{k-1}, r_k)$ for $k\geq1$;
\item A sequence $\{a_k\}$  where 
 \[a_k=\inf_{y\in B(x_{k-1}, r_k)}R(y, t_k)=R(x_k, t_k)\]
 and which satisfies
 \begin{equation*}
     \begin{split}
         a_{k-1}\geq &a_ke^{2(n-1)t_{k-1}}\left( \frac{e^{2(n-1)t_{k-1}}+1}{2} \right)^{-(1+n/2)}\\
         &-\frac{2c(n)}{(e^{2(n-1)t}-1)/2^k}\exp\left( -\frac{r_k^2}{D(e^{2(n-1)t}-1)/2^k} \right).
     \end{split}
 \end{equation*}  
\end{enumerate}

Consequently, we obtain
\begin{equation}\label{scal}
\begin{split}
    R(x, t)\geq& a_k e^{2(n-1)\sum_{i=0}^{k-1}t_i}\prod_{i=0}^{k-1}\left( \frac{e^{2(n-1)t_i}+1}{2}\right)^{-(1+n/2)}\\
    & -c(n)\sum_{i=1}^k\frac{1}{(e^{2(n-1)t}-1)/2^i}\exp\left( -\frac{r_i^2}{D(e^{2(n-1)t}-1)/2^i} \right).
\end{split}
\end{equation}
Since $t_k<\frac{3}{4}t_{k-1}<(\frac{3}{4})^kt$, we have  
\[ \sum_{i=0}^{\infty}t_i< \sum_{i=0}^{\infty}(\frac{3}{4})^it=4t, \]
and 
\[ \sum_{i=0}^{\infty}\ln\left( \frac{e^{2(n-1)t_i}+1}{2} \right)<\sum_{i=0}^{\infty}\frac{3(n-1)}{2}t_i<6(n-1)t, \]
both of which are convergent. Moreover,
\[ e^{2(n-1)\sum_{i=0}^{\infty}t_i}\prod_{i=0}^{\infty}\left( \frac{e^{2(n-1)t_i}+1}{2}\right)^{-(1+n/2)}<e^{(2-3n)(n-1)t}<1. \]

We now choose $\{r_i\}$ so that $\sum_{i=1}^{\infty}r_i$ converges and the second term of \eqref{scal} also converges.
Take $r_i=\left( \frac{e^{2(n-1)t}-1}{2^i} \right)^{\beta}$ for some $\beta\in (0, \frac{1}{2})$. Then 
\[\sum_{i=1}^{\infty}r_i=(e^{2(n-1)t}-1)^{\beta}\frac{1}{2^{\beta}-1}<C(n, \beta)t^{\beta},\]
and
\begin{align*}
   & \sum_{i=1}^{\infty}\frac{1}{(e^{2(n-1)t}-1)/2^i}\exp\left( -\frac{r_i^2}{D(e^{2(n-1)t}-1)/2^i} \right)\\
   =& \sum_{i=1}^{\infty}\frac{1}{(e^{2(n-1)t}-1)/2^i}\exp\left( -\frac{(e^{2(n-1)t}-1)^{2\beta-1}}{2^{(2\beta-1)i}D} \right)\\
   <&\sum_{i=1}^{\infty}D^{\gamma}\left( \frac{e^{2(n-1)t}-1}{2^i} \right)^{-1-(2\beta-1)\gamma}\\
   <&C(n, \beta)t^{-1-(2\beta-1)\gamma},
\end{align*}
for any $\gamma>0$ such that $-1-(2\beta-1)\gamma>0$.
Finally, we observe that
\begin{align*}
   & \lim_{k\to\infty}a_k=\lim_{k\to\infty}\inf_{y\in B(x_{k-1}, r_k)}R(y, t_k)\geq \lim_{t\to0}\inf_{B(x, \sum_{i=1}^{\infty}r_i)}R(\cdot, t)\\
   \geq &\inf_{C>0}(\liminf_{t\to0}(\inf_{y\in B(x, Ct^{\beta})}R(y, t))), 
\end{align*} 
  which yields the desired result.  
\end{proof}

\begin{definition}
    Let $M^n$ be a smooth manifold and $g$ be a $C^0$ Riemannian metric on $M$. Given $\beta\in (0, \frac{1}{2})$, we say that $g$ has scalar curvature bounded below by $\kappa\in \mathbb{R}$ in the $\beta$-weak sense at a point $x\in M$ with respect to $(\Phi, g_0, g_t)$ if there exist
    \begin{enumerate}
        \item a diffeomorphism $\Phi: U_x\to \Phi(U_x)$ from a neighborhood $U_x$ of $x$ onto an open subset $\Phi(U_x)\subset\mathbb{H}^n$;
        \item  a $C^0$ metric $g_0$ on $\mathbb{H}^n$ and a normalized $b$-flow $(g_t)_{t\in (0, T]}$ for $g_0$ satisfying Corollary \ref{cor3.3};
    \end{enumerate} 
    such that
    \begin{align*}
        g_0|_{\Phi(U_x)}=\Phi_*g,\\
        \inf_{C>0}\left(\liminf_{t\to0}\left(\inf_{B(\Phi(x), Ct^{\beta})}R(g_t)\right) \right)\geq\kappa.
    \end{align*}
\end{definition}

\begin{remark}
For a $C^0$ metric $g$ and fixed $(\Phi, g_0, g_t)$, the scalar curvature of $g$ bounded below by $\kappa$ in the $\beta$-weak sense at $x$  is denoted by $R_{C^0_{\beta}}(g)(x)\geq\kappa$.
\end{remark}

Corollary \ref{cor3.3} guarantees that for a continuous metric $g_0$ on $\mathbb{H}^n$ satisfying $||g_0-b||_{L^{\infty}(\mathbb{H}^n)}<\epsilon$, there exists a normalized $b$-flow $g(t)$ on $(0, T_0]$ with 
\[|e^{2(n-1)t}g(t)-b|<2\epsilon\]
for $t\in (0, T_0]$.
Thus, for sufficiently small $t$ and $\epsilon$, we have that $\frac{1}{2}b<g(t)<2b$. Below we will refine the estimates of Corollary \ref{cor3.3}.
Let $g(t)$ be the normalized $b$-flow \eqref{nrf} on $\mathbb{H}^n$.
By Lemma 2.1 of \cite{Shiwx}, equation $\eqref{nrf}$ with $h=b$ can be written by
\begin{align}\label{nbf}
    \partial_tg_{ij}=& g^{\alpha\beta}D_{\alpha}D_{\beta}g_{ij}-g^{\alpha\beta}g_{ip}b^{pq}R_{j\alpha q\beta}(b)-g^{\alpha\beta}g_{jp}b^{pq}R_{i\alpha q\beta}(b)-2(n-1)g_{ij}\notag\\
    &+\frac{1}{2}g^{\alpha\beta}g^{pq}(D_ig_{p\alpha}D_jg_{q\beta}+2D_{\alpha}g_{jp}D_{q}g_{i\beta}-2D_{\alpha}g_{jp}D_{\beta}g_{iq}-2D_jg_{p\alpha}D_{\beta}g_{iq}\notag\\&-2D_ig_{p\alpha}D_{\beta}g_{jq}).
\end{align}

Write
\[g_{ij}(t)=h_{ij}(t)+b_{ij}, \quad\quad\quad g^{ij}(t)=b^{ij}+f^{ij}(t),\] 
where all indices are lowered and raised by the metric $b$.
Then $\delta_i^j=g_{ik}g^{kj}$ implies 
\[ f^{ij}=-h^{ij}-h^i_kf^{kj}. \]
Using (2.3) of \cite{gicquaud2025}, we have $|f|<C|h|$.
So we can calculate 
\begin{align*}
   g^{\alpha\beta}D_{\alpha}D_{\beta}g_{ij}=\Delta g_{ij}+f^{\alpha\beta}D_{\alpha}D_{\beta}g_{ij}=\Delta g_{ij}+h*D^2h,
\end{align*}
and
\begin{align*}
    &g^{\alpha\beta}g_{ip}b^{pq}R_{j\alpha q\beta}(b)=(b^{\alpha\beta}+f^{\alpha\beta})(h_{ip}+b_{ip})R_{j\alpha q\beta}(b)\\
    =&-(n-1)b_{ij}-nh_{ij}+b^{kl}h_{kl}b_{ij}+h*h.
\end{align*}
Therefore, \eqref{nbf} is equivalent to
\begin{equation}\label{nbf1}
    \partial_th_{ij}=\Delta h_{ij}+2h_{ij}-2b^{kl}h_{kl}b_{ij}+Q[h]:=-Lh_{ij}+Q[h],
\end{equation}
where 
\[Lh_{ij}=-\Delta h_{ij}-2h_{ij}+2b^{kl}h_{kl}b_{ij}\]
is the linear operator and 
\[ Q[h]=Q^0[h]+D Q^1[h], \]
\[ Q^0[h]=h*h+Dh*Dh, \]
\[  Q^1[h]=h *D h. \]
Here $\Delta$ and $D$ denote the Laplacian operator and the covariant derivative with respect to the background metric $b$. The next theorem shows that exponential decay is preserved for the $b$-flow.

\begin{thm}\label{thm3.7}
    Let \((\mathbb H^n,b)\) be the hyperbolic space of sectional curvature \(-1\), let \(s(x)=d_b(0,x)\), and let \(g(t)\) be the normalized \(b\)-flow from a continuous initial metric \(g_0\) in Corollary \ref{cor3.3}. Put \(h(t)=g(t)-b\). There are \(\varepsilon_*>0\), \(T_*>0\), and \(C<\infty\), depending only on \(n,\tau\), such that if
\[
\|g_0-b\|_{L^\infty(\mathbb{H}^n)}\le \varepsilon_*,\qquad |g_0-b|_b\le C_0e^{-\tau s},
\]
then for every \(0<t\le T_*\) and every \(x\in\mathbb H^n\),
\[
|h(t,x)|_b+t^{1/2}|Dh(t,x)|_b+t|D^2h(t,x)|_b
\le C C_0e^{-\tau s(x)} .
\]
\end{thm}

\begin{proof}
    By Corollary \ref{cor3.3}, $g(t)$ is smooth and close to $b$. Let $u(t)=|h(t)|^2$. 
Lemma 2.2 in \cite{Schnurer} gives
\begin{equation}\label{part}
   \partial_t u\le a^{ij}D_iD_j u+C_1u,\qquad a^{ij}=g^{ij}, 
\end{equation}
with \(a^{ij}\) uniformly elliptic with respect to \(b\), because \(\|g-b\|_\infty\) remains small on a short interval Corollary 3.3.

Choose a smooth proper function \(q\) on \(\mathbb H^n\) with \(q=s+O(1)\), \(|Dq|\le C\), \(|D^2q|\le C\); for example smooth \((1+s^2)^{1/2}\) near the origin. For \(R\ge2\), choose \(\chi_R:\mathbb R\to\mathbb R\) with
\[
\chi_R(r)=r\quad(r\le R),\qquad \chi_R(r)=R+1\quad(r\ge R+2),
\]
and \(0\le\chi_R'\le1\), \(|\chi_R''|\le C\). Put
\[
W_R=e^{\tau\chi_R(q)}.
\]
Then \(W_R\simeq \min(e^{\tau s},e^{\tau R})\), and
\[
|D\log W_R|+|D^2\log W_R|\le C(n,\tau)
\]
uniformly in \(R\).

Set \(v_R=W_R^2u\). Since \(W_R\) is time-independent, \eqref{part} leads to
\[
\partial_t v_R
\le a^{ij}D_iD_jv_R-4a^{ij}D_i(\log W_R)D_jv_R+C_2v_R,
\]
where \(C_2=C_2(n,\tau)\). The coefficients and drift are bounded uniformly in \(R\). 
Applying the noncompact maximum principle (Lemma C.1 in \cite{Schnurer}) is justified as follows. For \(z=e^{-C_2t}v_R\), one has 
\[ \partial_tz\leq L_Rz, \]
where $L_R=a^{ij}D_iD_j-4a^{ij}D_i(\log W_R)D_j$.
So $|L_Rq|\le C(n, \tau)$ uniformly in $R$. Take $A>\sup_{\mathbb{H}^n}|L_Rq|$ and let 
$z_{\delta}=z-\delta(q+At)$.
Then 
\[ (\partial_t-L_R)z_{\delta}=(\partial_t-L_R)z_{\delta}-\delta(A-L_Rq)<0. \]
Since $z_{\delta}(x, t)\rightarrow-\infty$ as $x\to \infty$,
by the maximum principle, we get
\[ \sup_{\mathbb{H}^n}z_{\delta}(\cdot, t)\leq\sup_{\mathbb{H}^n}z_{\delta}(\cdot, 0)\le \sup_{\mathbb{H}^n}z(\cdot, 0).
\]
Letting $\delta\to0$ gives
\[
\sup_{\mathbb H^n}v_R(t)\le e^{C_2t}\sup_{\mathbb H^n}v_R(0).
\]
The initial weighted bound gives
\[
\sup_{\mathbb H^n} W_R^2|g_0-b|^2
\le C(n,\tau)C_0^2,
\]
uniformly in \(R\), because \(q=s+O(1)\). Letting \(R\to\infty\) yields
\[
|h(t,x)|\le C C_0e^{-\tau s(x)}.
\]

For derivatives, fix \(0<t\le T_*\) and work on
\[
Q=B_b(x,2\sqrt t)\times[t/2,t].
\]
On \(Q\), \(e^{-\tau s(y)}\le C(n,\tau,T_*)e^{-\tau s(x)}\). The local parabolic interior estimates for the uniformly parabolic $b$-blow, after rescaling \(Q\) to a unit cylinder and using the small \(L^\infty\) closeness to absorb the \(h*D^2h\) principal perturbation, give
\[
\sqrt t\,|Dh(t,x)|+t|D^2h(t,x)|
\le C\sup_Q |h|.
\]
So we get the derivative estimates.

    
\end{proof}

\begin{thm}\label{thm3.8}
    Under the hypotheses of Theorem \ref{thm3.7}, after possibly reducing \(T_*\), we have
\[
R^{-n/2}\|Dh\|_{L^2(B_b(x,R)\times(0,R^2))}
\le C C_0e^{-\tau s(x)}
\]
for all \(x\in\mathbb H^n\) and \(0<R\le\sqrt{T_*}\).
\end{thm}

\begin{proof}
    By Lemma 2.2 of \cite{Schnurer}, let $u=|g(t)-b|^2=|h|^2$ satisfying
    \begin{equation}\label{uht}
        \partial_t u
\le
g^{ij}D_iD_j u
-c_0|Dh|^2
+C u ,
    \end{equation}
where $c_0>0$.
Choose a cutoff function $\zeta\in C_c^\infty(B_b(x,2R))$ such that
$$
0\le \zeta\le 1,\qquad
\zeta\equiv 1 \ \text{on } B_b(x,R),\qquad
|D\zeta|\le \frac{C}{R}.
$$
Multiplying \eqref{uht} by $\zeta^2$ and integrating with respectto $d\mu_b$, we get
$$
\frac{d}{dt}\int_{\mathbb H^n}\zeta^2u\,d\mu_b
\le
\int_{\mathbb H^n}\zeta^2 g^{ij}D_iD_j u\,d\mu_b
-c_0\int_{\mathbb H^n}\zeta^2|Dh|^2\,d\mu_b
+C\int_{\mathbb H^n}\zeta^2u\,d\mu_b .
$$
For the first term of right hand side, integration by parts gives
$$
\begin{aligned}
\int_{\mathbb H^n}\zeta^2 g^{ij}D_iD_j u\,d\mu_b
&=
-\int_{\mathbb H^n}D_i(\zeta^2g^{ij})D_j u\,d\mu_b  \\
&=
-2\int_{\mathbb H^n}\zeta g^{ij}D_i\zeta D_j u\,d\mu_b
-\int_{\mathbb H^n}\zeta^2D_i g^{ij}D_j u\,d\mu_b .
\end{aligned}
$$
Since
$
|Du|\le C|h||Dh|,
$
we have
$$
\begin{aligned}
\left|
2\int_{\mathbb H^n}\zeta g^{ij}D_i\zeta D_j u\,d\mu_b
\right|
&\le
C\int_{\mathbb H^n}\zeta |D\zeta|\,|h|\,|Dh|\,d\mu_b  \\
&\le
\frac{c_0}{4}\int_{\mathbb H^n}\zeta^2|Dh|^2\,d\mu_b
+
C\int_{\mathbb H^n}|D\zeta|^2|h|^2\,d\mu_b .
\end{aligned}
$$
Moreover, since $Dg=Dh$,
$$
\begin{aligned}
\left|
\int_{\mathbb H^n}\zeta^2D_i g^{ij}D_j u\,d\mu_b
\right|
&\le
C\int_{\mathbb H^n}\zeta^2 |Dh|\,|h|\,|Dh|\,d\mu_b  \\
&=
C\int_{\mathbb H^n}\zeta^2 |h|\,|Dh|^2\,d\mu_b  \\
&\le
C\|h\|_{L^\infty}
\int_{\mathbb H^n}\zeta^2|Dh|^2\,d\mu_b .
\end{aligned}
$$
By Theorem \ref{thm3.7}, $\|h\|_{L^\infty}\le \varepsilon_*$, and taking
$\varepsilon_*>0$ sufficiently small, the last term can be absorbed
into the left-hand side. Hence
$$
\frac{d}{dt}\int_{\mathbb H^n}\zeta^2|h|^2\,d\mu_b
+
c\int_{\mathbb H^n}\zeta^2|Dh|^2\,d\mu_b
\le
C\int_{\mathbb H^n}|D\zeta|^2|h|^2\,d\mu_b
+
C\int_{\mathbb H^n}\zeta^2|h|^2\,d\mu_b .
$$
Since $|D\zeta|\le C/R$, this implies
$$
\frac{d}{dt}\int_{\mathbb H^n}\zeta^2|h|^2\,d\mu_b
+
c\int_{\mathbb H^n}\zeta^2|Dh|^2\,d\mu_b
\le
C(R^{-2}+1)
\int_{B_b(x,2R)}|h|^2\,d\mu_b .
$$
Integrating from $0$ to $R^2$, and using
$\zeta\equiv 1$ on $B_b(x,R)$, we obtain
$$
\int_0^{R^2}\int_{B_b(x,R)}|Dh|^2\,d\mu_b\,dt
\le
C\int_{B_b(x,2R)}|h_0|^2\,d\mu_b
+
C(R^{-2}+1)
\int_0^{R^2}\int_{B_b(x,2R)}|h|^2\,d\mu_b\,dt .
$$
On $B_b(x,R)\times(0,R^2)$, $R\leq\sqrt{T_*}$, we have \(e^{-\tau s(y)}\le C(n,\tau,T_*)e^{-\tau s(x)}\).
By Theorem \ref{thm3.7}, 
\[ \int_{B_b(x,2R)}|h_0|^2\,d\mu_b
+
R^{-2}\int_0^{R^2}\int_{B_b(x,2R)}|h|^2\,d\mu_b\,dt\leq CC_0^2e^{-2\tau s(x)}R^n, \]
which finishes the proof.
\end{proof}

\vspace{.2in}

\section{$C^0$ local mass under the normalized Ricci-DeTurck flow}

\begin{proposition}\label{prop51}
    Let $(M, g)$ be a $C^0_{\tau}$-asymptotically hyperbolic manifold with a fixed coordinate chart at infinity $\Phi$ and  $\tau>\frac{n}{2}$.
    Let $g_0$ be a continuous metric on $\mathbb{H}^n$ such that $g_0=\Phi_*g$ on $A(0, 0.85r, 11.5r)$ for $r>0$ sufficiently large and 
    \[ ||g_0-b||_{L^{\infty}(\mathbb{H}^n)}<\epsilon, \]
    for some $\epsilon<1$.
    Let $g_t=g(t)$ be the normalized $b$-flow as stated in Theorem \ref{thm3.7}.
    
       Suppose $R_{C^0_{\beta}}(g)(x)\geq-n(n-1)$, for $\beta\in (0, \frac{1}{2})$, and for any $\Phi(x)\in A(0, 0.85r, 11.5r)$. There exists a constant $C(n, \beta, \tau)$ such that, if $\varphi^1, \varphi^2$ are any two positive  smooth functions  with nonzero integrals over $(0.9, 1.1)$, then for all $r'\in [\frac{1.1}{0.9}r, 10r]$,  we have
    \[ M_{C^0}(g_{t}, \varphi^1, r')-M_{C^0}(g_{t}, \varphi^2, r)\geq -C(n, \beta,\tau)(t^{-1}e^{0.9(n-2\tau)r}+e^{11nr}t^{\lambda}), \]
    for $t$ sufficiently small and any $\lambda>0$.
\end{proposition}

\begin{proof}
    Let $s=|x|_{b}=\ln\left(\frac{1+|x|_{\delta}}{1-|x|_{\delta}}\right)$ be the geodesic distance in $\mathbb{H}^n$ and $V^0=\rho^{-1}-1=\cosh s$. 
    For $s$ sufficiently large, we have $|V^0|+|dV^0|=O(\rho^{-1})$.
    By \eqref{grad} and Theorem \ref{thm3.7}, $|e|=O(\rho^{\tau})\le  Ce^{-\tau s}$ on $A(0, 0.9r, 11r)$, we have, for $e_t=g_t-b$,
    \[ |D e_t|\leq Ct^{-\frac{1}{2}}e^{-\tau s}, \]
    \[ |D^2e_t|\leq Ct^{-1}e^{-\tau s}, \]
    and 
    \[ |e_t|\leq Ce^{-\tau s}. \]
    By Definition \ref{c2mass}, since $g_t$ is smooth for $t>0$, there exist $l_1, l_2\in (0.9, 1.1)$ such that 
    \[ M_{C^0}(g_{t}, \varphi^1, r')= \frac{\int_{0.9}^{1.1}\varphi^1(l)M_{C^2}(g_t,  r'l, V^0)dl}{\int_{0.9}^{1.1}\varphi^1(l)dl}=M_{C^2}(g_t, r'l_1, V^0),\]
    \[ M_{C^0}(g_t, \varphi^2, r)= M_{C^2}(g_t, rl_2, V^0). \]
We see that $r'l_1\in (1.1r, 11r)$, $rl_2\in (0.9r, 1.1r)$. 
Hence, by Lemma \ref{c2scalar}, we get, 
\begin{align*}
    &M_{C^0}(g_t, \varphi^1, r')-M_{C^0}(g_t, \varphi^2, r)\\
    =& M_{C^2}(g_t, r'l_1, V^0)-M_{C^2}(g_t, rl_2, V^0)\\
    =& \int_{A(0, rl_2, r'l_1)}[V^0(R_{g_t}+n(n-1))+Q]d\mu_b.
\end{align*}
where we have that
\begin{align*}
   & Q\lesssim (|V^0|+|dV^0|)(|De_t|^2+|e_t|^2)+|V^0||e_t||D^2e_t|\\
    \lesssim & \rho^{-1}(t^{-1}\rho^{2\tau}+\rho^{2\tau})+\rho^{\tau-1}\rho^{\tau}t^{-1}\\
    \lesssim &  e^{(1-2\tau)s}t^{-1},
\end{align*}  
and
\begin{equation*}
    \begin{split}
    &\int_{A(0, rl_2, r'l_1)}Qd\mu_b\\
    \lesssim&t^{-1}\int_{0.9r}^{11r}e^{(n-2\tau)s}ds\lesssim t^{-1}e^{0.9(n-2\tau)r}.
    \end{split}
\end{equation*} 
By Theorem \ref{scalart}, for $\beta\in (0, \frac{1}{2})$, we have
\[ R(g_t)\geq -n(n-1)-C(n, \tau, \beta)t^{\lambda}, \]
for any $\lambda>0$.
   Therefore,  we obtain that
   \begin{align*}
       &M_{C^0}(g_t, \varphi^1, r')-M_{C^0}(g_t, \varphi^2, r)\\
       \geq & -C (t^{-1}e^{0.9(n-2\tau)r}+e^{11nr}t^{\lambda}),
   \end{align*}
   where $C=C(n, \tau, \beta)$ and $t$ sufficiently small.
\end{proof}

\begin{thm}\label{thm5.2}
    For any $r>r_0(n)$, let $\varphi: \mathbb{R}\to\mathbb{R}_{\geq0}$ be a smooth cutoff function with $\text{supp} \varphi\subset(a, b)\subset\subset(0.9, 1.1)$, where $a>0.9$, $b<1.1$. 
    For $\theta\in (0, \frac{2d_{ab}^2}{n})$, there exists a unique smooth function 
    $\phi_{\theta}:\mathbb{R}\times [0, \theta]\to \mathbb{R}$ satisfying
    \begin{equation}\label{laps}
        \begin{cases}
            (\partial_t+\Delta)\phi_{\theta}(s, t)=f(s)\phi_{\theta}(s, t), \quad\quad&\text{for}\quad(x, t)\in \mathbb{H}^n\times (0, \theta);\\
            \phi_{\theta}(s, \theta)=\cosh{(s)}\varphi(\frac{s}{r}), \quad \quad &\text{for}\quad x\in \mathbb{H}^n; 
        \end{cases}
    \end{equation}
    where $s=|x|_b$ is the geodesic distance and $f(s)$ is a smooth function such that
    \begin{equation*}
        f(s)\begin{cases}
            = 2n-2+(n-1)(\sinh s)^{-2}-2(\cosh s)^{-2}, \quad\quad &\text{for }\quad s\geq 0.9r;\\
             \in [2n-4, 2n-1], \quad \quad &\text{for}\quad s<0.9r;
        \end{cases}
    \end{equation*}
    and we require that 
    \[ |f'(s)|\leq Cr^{-1},\quad |f''(s)|\leq Cr^{-2}, \]
    for $s<0.9r$
    and $ d_{ab}=\min\{ a-0.9, 1.1-b \}$.
For any $(x, t)\in A(0, 0.9r, 1.1r)\times (0, \theta]$, we have $\phi_{\theta}(s, t)\geq0$ and the estimates
\[ |\phi_{\theta}|\leq C(n, \varphi)e^{s}, \quad |\phi'_{\theta}|\leq C(n, \varphi)e^{s}, \quad |\phi''_{\theta}|\leq C(n, \varphi)e^{s}, \]
where prime denotes the partial derivative with respect to the geodesic distance.
    
     Moreover, for $(x, t)\in \partial A(0, 0.9r, 1.1r)\times [0, \theta]$, we have 
    \begin{equation*}
    |\phi_\theta(x,t)|
\le C(n,\varphi)\theta^{-n/2}
\exp\left(-\frac{d_{ab}^2r^2}{c\theta}\right)e^{1.1nr},
    \end{equation*}
    \begin{equation*}
    |\partial_s\phi_\theta(x,t)|
\le C(n,\varphi)\theta^{-(n+1)/2}
\exp\left(-\frac{d_{ab}^2r^2}{c\theta}\right)e^{1.1nr},
    \end{equation*}
    \begin{equation*}
        |\partial_s^2\phi_\theta(x,t)|
\le C(n,\varphi)\theta^{-(n+2)/2}
\exp\left(-\frac{d_{ab}^2r^2}{c\theta}\right)e^{1.1nr}.
    \end{equation*}
    where $c$ is a constant. 
\end{thm}

\begin{proof}
    Since for $r>r_0(n)$,
    \[f(s)< 2n-2+(n-1)(\sinh(0.9r))^{-2}<2n-1\]
    and 
    \[f(s)>2n-4,\]
    by Theorem 24.40 of \cite{Chow}, there is a unique smooth positive solution $H(x,t; y, v)$ to 
    \begin{equation*}
        \begin{cases}
            (\partial_t-\Delta_x+f(|x|_b))H(x, t; y,v)=0;  \\
            \lim_{t\to v}H(x, t; y, v)=\delta_y;
        \end{cases}
    \end{equation*}
    for $t>v$. For fixed $(x, t)\in\mathbb{H}^n\times(0, \theta) $, there holds
    \begin{equation*}
        \begin{cases}
            (\partial_v+\Delta_y-f(|y|_b))H(x, t; y, v)=0;\\
             \lim_{v\to t}H(x, t; y,v)=\delta_x.
        \end{cases}
    \end{equation*}
    For brevity, we write $\phi_{\theta}(x, t)=\phi_{\theta}(|x|_b,t)$.
    Then we have
    \[ \phi_{\theta}(x, t)=\int_{\mathbb{H}^n}H(y, \theta; x, t)\cosh{(|y|_b)}\varphi(\frac{|y|_b}{r})d\mu_b(y). \]
    Since $\cosh(s)\varphi(\frac{s}{r})$  and $f$ are radial, uniqueness implies that
$\phi_\theta$ is radial. Using barrier $Ae^{B\tau}\cosh s$ and the maximum principle gives
$$|\phi_\theta|\le C(n,\varphi)e^{s}.$$

For $p=\partial_s\phi_\theta$, differentiating the radial equation gives
$$
(\partial_t+\Delta_b)p
=\bigl(f+(n-1)\operatorname{csch}^2s\bigr)p+f'\phi_\theta .
$$
The coefficients are bounded and $p(\cdot,\theta)=O(e^s)$, hence the maximum principle gives $|p|\le C(n,\varphi)e^{s}$. Similarly, we get the same estimate for the second derivative.
By Theorem 26.25 of \cite{Chow} and Corollary 2.7 of \cite{paula19}, for $k=0, 1, 2, 3,$
\[ H(y, \theta; x, t)\leq C(n)(\theta-t)^{-\frac{n}{2}}\exp\left(-\frac{d_b^2(x, y)}{c(\theta-t)} \right), \]
 \[ |D_kH(y, \theta; x, t)|\leq C(n)(\theta-t)^{-\frac{n+k}{2}}\exp\left(-\frac{d_b^2(x, y)}{c(\theta-t)} \right), \]
for $(x, t)\in \partial A(0, 0.9r, 1.1r)\times[0, \theta]$,
there holds
\begin{align*}
    &\phi_{\theta}(|x|_b, t)\leq C(n)\int_{\mathbb{H}^n}(\theta-t)^{-\frac{n}{2}}\exp\left(-\frac{d_b^2(x, y)}{c(\theta-t)} \right)\cosh{(|y|_b)}\varphi(\frac{|y|_b}{r})d\mu_b(y)\\
    \leq&  C(n, \varphi)e^{1.1r}\int_{A(0, ar, br)}(\theta-t)^{-\frac{n}{2}}\exp\left(-\frac{d_b^2(x, y)}{c(\theta-t)} \right)d\mu_b(y)\\
     \leq&  C(n, \varphi)e^{1.1r}(\theta-t)^{-\frac{n}{2}}\exp\left(-\frac{d_{ab}^2r^2}{c(\theta-t)} \right)\text{vol}(A(0, 0.9r, 1.1r))\\
     \leq & C(n, \varphi)e^{1.1r}\theta^{-\frac{n}{2}}\exp\left(-\frac{d_{ab}^2r^2}{c\theta} \right)e^{1.1(n-1)r},
\end{align*}
where in the last line we used the fact that $(\theta-t)^{-\frac{n}{2}}\exp\left(-\frac{d_{ab}^2}{c(\theta-t)} \right)$ is decreasing in $t$ for $t<\theta<\frac{2d_{ab}^2}{n}$.
Similarly, we get  the other two derivative estimates.
\end{proof}

\begin{lemma}
    For any function $u\in C^{\infty}(\mathbb{H}^n)$ with 
    \[ u(x)=u(|x|_b)=u(s), \]
    where we denote $s=|x|_b$ as the geodesic distance, we have 
    \[ \Delta u(s)=u''+(n-1)u'\frac{\cosh s}{\sinh s}, \]
    where the prime denotes the derivative with respect to $s$.
\end{lemma}

\begin{proof}
    Since $b=\rho^{-2}\delta$ is conformally flat, so
\[ \Delta u=\rho^2\Delta_{\delta}u+(n-2)\rho x^i\partial_iu. \]
As $\partial_ks=\frac{x^k}{\rho^2\sinh s}$, direct computation gives
\[ x^i\partial_iu=u'\sinh s, \]
\[ \Delta_{\delta}u=\rho^{-2}u''+u'\rho^{-1}[\sinh s +(n-1)(\rho\sinh s)^{-1}]. \]
Then we get
\[ \Delta u=u''+u'\rho[\sinh s +(n-1)(\rho\sinh s)^{-1}]+(n-2)u'\rho\sinh s, \]
and since $1+\rho\sinh^2s=\cosh s$, so
\[ \Delta u=u''+(n-1)u'\frac{\cosh s}{\sinh s}. \]
    
\end{proof}

\begin{corollary}\label{cor5.4}
     For any function $u\in C^{\infty}(\mathbb{H}^n)$ with 
    \[ u(x)=u(|x|_b)=u(s), \] 
    we have 
    \[ \Delta u'=(\Delta u)'+(n-1)u'(\sinh s)^{-2}. \]
\end{corollary}

\begin{thm}\label{thm5.5}
 Let $(M, g)$ be a $C^0_{\tau}$-asymptotically hyperbolic manifold with a fixed coordinate chart at infinity $\Phi$ and $\tau>\frac{n}{2}$.
 Let $\varphi:\mathbb{R}\to\mathbb{R}_{\geq0}$ be a smooth cutoff function with $\int_{0.9}^{1.1}\varphi(l)dl\not=0$ and $supp\varphi\subset(a, b)\subset\subset(0.9, 1.1)$.
Let $g_t$ be the normalized $b$-flow as in Proposition \ref{prop51}.
Let $\phi(x, t)=\phi(|x|_b, t)\in C^{\infty}(\mathbb{H}^n\times[0, \theta])$ be a function satisfying
\begin{equation*}
    \begin{cases}
        (\partial_t+\Delta)\phi=[2n-2+(n-1)(\sinh s)^{-2}-2(\cosh s)^{-2}]\phi, &\quad\text{on}\quad A(0, 0.9r, 1.1r)\times (0, \theta);\\
        \phi(s, \theta)=\cosh s\varphi(\frac{s}{r}), &\quad\text{for}\quad s\in[0.9r, 1.1r];
    \end{cases}
\end{equation*}
where $\theta<T_*$.
Then, for any $r>r_0(n)$ sufficiently large and $\theta$ sufficiently small, we have 
\[ \int_0^{\theta}\left| \frac{d}{dt}\left[M_{C^0}(g_t, \varphi_{\theta}(\cdot, t), r)r\int_{0.9}^{1.1}\varphi_{\theta}(l, t)dl\right] \right|dt\leq  C(n, \varphi)e^{0.9(n-2\tau)r}, \]
where  
\begin{equation}\label{varphithe}
    \begin{cases}
        \varphi_{\theta}(s, t)=\varphi_{\theta}(|x|_b, t)=(\cosh (rs))^{-1}\phi(sr, t),\quad&\text{for}\quad (s, t)\in (0.9, 1.1)\times[0, \theta];\\
        \varphi_{\theta}(s, \theta)=(\cosh (rs))^{-1}\phi(sr, \theta)=\varphi(s), \quad&\text{for}\quad s\in (0.9, 1.1).
    \end{cases}
\end{equation}
\end{thm}

\begin{proof}
Note that 
\[ \phi(s, t)=\cosh s\varphi(\frac{s}{r}, t) \]
implies 
\[ \phi'(s, t)=\cosh sr^{-1}\varphi'(\frac{s}{r}, t)+\sinh s\varphi(\frac{s}{r}, t), \]
where we omit the subscript $\theta$.
From Definition \ref{defc0}, we write
\[ M_{C^0}(g_t,\varphi_{\theta}(\cdot, t), r)r\int_{0.9}^{1.1}\varphi_{\theta}(l, t)dl=I+II+III, \]
where
\begin{align*}
I=&\int_{\partial A(0, 0.9r, 1.1r)}\cosh s\varphi(\frac{s}{r}, t)(\nu^i\nu^j-b^{ij})e_{ij}d\mu_b,\\
    II=& \int_{A(0, 0.9r, 1.1r)}\left[ \cosh s\varphi'\left( \frac{s}{r}, t \right)r^{-1}+\varphi\left( \frac{s}{r}, t \right)\left(n\sinh s +\frac{n-2}{\sinh s}\right) \right]b^{ij}e_{ij}d\mu_b\\
    =& \int_{A(0, 0.9r, 1.1r)}\left[\phi'(s, t)+\left((n-1)\frac{\sinh s}{\cosh s}+\frac{n-2}{\sinh s\cosh s}\right)\phi(s, t)\right]tr_bed\mu_b,\\
    III=& \int_{A(0, 0.9r, 1.1r)}\left[ \varphi\left(\frac{s}{r}, t\right)\left(\frac{1}{\sinh s}-\sinh s\right)-\cosh s\varphi'\left( \frac{s}{r}, t \right)r^{-1} \right]\frac{x^ix^j}{\sinh^2 s}e_{ij}d\mu_b\\
    =& \int_{A(0, 0.9r, 1.1r)}[-\phi'(s, t)+(\sinh s \cosh s)^{-1}\phi(s, t)](\sinh s)^{-2}x^ix^je_{ij}d\mu_b.
\end{align*}
By the normalized $b$-flow equation \eqref{nbf1}
\[ \partial_t e_{ij}=\Delta e_{ij}+2e_{ij}-2tr_beb_{ij}+Q, \]
where 
\[ Q=e*e+De*De+D(e*De), \]
let 
\[u(s)=(n-1)\frac{\sinh s}{\cosh s}+\frac{n-2}{\sinh s\cosh s}=\frac{\sinh s}{\cosh s}+(n-2)\frac{\cosh s}{\sinh s}\] and we can calculate that
\begin{align*}
    &\frac{d}{dt}II=\int_{A(0, 0.9r, 1.1r)}[(\partial_t\phi'+u\partial_t\phi)tr_be+(\phi'+u\phi)b^{ij}\partial_te_{ij} ]d\mu_b\\
    =&\int_{A(0, 0.9r, 1.1r)}[(\partial_t\phi'+u\partial_t\phi)tr_be+(\phi'+u\phi)b^{ij}(\Delta e_{ij}+2e_{ij}-2tr_beb_{ij}+Q) ]d\mu_b\\
    =& \int_{A(0, 0.9r, 1.1r)}[(\partial_t+\Delta)\phi'+u(\partial_t+\Delta)\phi+2(1-n)\phi'+2(1-n)u\phi\\
    &\quad\quad\quad\quad\quad\quad  +\phi \Delta u+2\phi'u']tr_bed\mu_b+C\\
    =& \int_{A(0, 0.9r, 1.1r)}[(\partial_t+\Delta)\phi'+u(\partial_t+\Delta)\phi+2\phi'(u'+1-n)\\
    &\quad\quad\quad\quad\quad\quad  +\phi(\Delta u+2(1-n)u)]tr_bed\mu_b+C ,
\end{align*}
where 
\begin{align*}
    C=&\int_{A(0, 0.9r, 1.1r)}(\phi'+u\phi)Q d\mu_b\\
    &+\int_{\partial A(0, 0.9r, 1.1r)}[(\phi'+u\phi)b^{ij}D_{\nu}e_{ij}-D_{\nu}(\phi'+u\phi)tr_be]d\mu_b.
\end{align*}

Let 
\[ X=\frac{x}{\sinh s}=\frac{x^i\partial_i}{\sinh s}, \]
\[ f(s, t)=-\phi'+(\sinh s \cosh s)^{-1}\phi \]
And direct computation gives
\begin{align*}
    \frac{d}{dt}III
     =\int_{A(0, 0.9r, 1.1r)}&[(-\partial_t\phi'+(\sinh s \cosh s)^{-1}\partial_t\phi)(\sinh s)^{-2}x^ix^je_{ij}\\
     &+f(\sinh s)^{-2}x^ix^j(\Delta e_{ij}+Q)\\
    &+2f(\sinh s)^{-2}x^ix^je_{ij}-2ftr_be ]d\mu_b.
\end{align*}

 We obtain by integration by parts
\begin{align*}
    &\int_{A(0, 0.9r, 1.1r)}f(\sinh s)^{-2}x^ix^j\Delta e_{ij}d\mu_b=\int_{A(0, 0.9r, 1.1r)}fX^iX^j\Delta e_{ij}d\mu_b\\
    =& \int_{A(0, 0.9r, 1.1r)}b^{kl}D_kD_l(fX^iX^j)e_{ij}d\mu_b\\
    &+\int_{\partial A(0, 0.9r, 1.1r)}[fX^iX^jb^{kl}D_le_{ij}\nu_k-b^{kl}D_l(fX^iX^j)\nu_ke_{ij}]d\mu_b.
\end{align*}
Since $D_kx=\cosh s\partial_k$, we have
\[ D_kX=(\sinh s)^{-1}\cosh s\partial_k-(\sinh s)^{-3}\cosh s\rho^{-2}x^kx, \]
\[ D_kX^i=(\sinh s)^{-1}\cosh s\delta_{ik}-(\sinh s)^{-3}\cosh s\rho^{-2}x^kx^i. \]
So
\begin{align*}
    &D_lD_kX=D_l [(\sinh s)^{-1}\cosh s\partial_k-(\sinh s)^{-3}\cosh s\rho^{-2}x^kx]\\
    =& -\partial_l((\sinh s)^{-3}\cosh s\rho^{-2}x^k)x-(\sinh s)^{-3}(\cosh s)^2\rho^{-2}x^k\partial_l\\
    &+\partial_l((\sinh s)^{-1}\cosh s)\partial_k+(\sinh s)^{-1}\cosh s\Gamma_{lk}^i\partial_i
\end{align*}
and 
\begin{align*}
    b^{kl}D_lD_kX^i
    =&  (1-n)(\sinh s)^{-3}(\cosh s)^2x^i.
\end{align*}
Since
\begin{align*}
   & b^{kl}D_lfD_kX^i=\rho^2f'\rho^{-2}(\sinh s)^{-1}x^lD_lX^i=0
\end{align*}
and
\begin{align*}
    & b^{kl}D_kX^iD_lX^j\\
    =& \rho^{2}[(\sinh s)^{-1}\cosh s\delta_{ik}-(\sinh s)^{-3}\cosh s\rho^{-2}x^kx^i][(\sinh s)^{-1}\cosh s\delta_{jk}-(\sinh s)^{-3}\cosh s\rho^{-2}x^kx^j]\\
    =& -(\sinh s)^{-4}(\cosh s)^2x^ix^j+(\sinh s)^{-2}(\cosh s)^2b^{ij},
\end{align*}
we obtain
\begin{align*}
    &b^{kl}D_kD_l(fX^iX^j)e_{ij}\\
    =&\Delta fX^iX^je_{ij}+4b^{kl}D_lfD_kX^iX^je_{ij}+2fb^{kl}D_kD_lX^iX^je_{ij}+2fb^{kl}D_kX^iD_lX^je_{ij}\\
    =& \Delta fX^iX^je_{ij}-2nf(\sinh s)^{-4}(\cosh s)^2x^ix^je_{ij}+2f(\sinh s)^{-2}(\cosh s)^2tr_be.
\end{align*}

Hence, we get
\begin{align*}
    \frac{d}{dt}III
     =\int_{A(0, 0.9r, 1.1r)}&[(-\partial_t\phi'+(\sinh s \cosh s)^{-1}\partial_t\phi+\Delta f+2f\\
     &
     -2nf(\sinh s)^{-2}(\cosh s)^2)(\sinh s)^{-2}x^ix^je_{ij}\\
     &+2f(\sinh s)^{-2}tr_be]d\mu_b+D,
\end{align*}
where 
\begin{align*}
    D=&\int_{A(0, 0.9r, 1.1r)}fX^iX^jQd\mu_b\\
    &+\int_{\partial A(0, 0.9r, 1.1r)}[fX^iX^jb^{kl}D_le_{ij}\nu_k-b^{kl}D_l(fX^iX^j)\nu_ke_{ij}]d\mu_b.
\end{align*}

Then we write
\begin{align*}
    \frac{d}{dt}(II+III)
     =\int_{A(0, 0.9r, 1.1r)}[Atr_be+B(\sinh s)^{-2}x^ix^je_{ij}]d\mu_b+C+D,
\end{align*}
where 
\begin{align*}
    A=&(\partial_t+\Delta)\phi'+u(\partial_t+\Delta)\phi+2\phi'(u'+1-n)+\phi(\Delta u+2(1-n)u)+2f(\sinh s)^{-2}\\
    =&(\partial_t+\Delta)\phi'+u(\partial_t+\Delta)\phi+\phi'\alpha+\phi\beta,\\
    B=&-\partial_t\phi'+(\sinh s \cosh s)^{-1}\partial_t\phi+\Delta f+2f
     -2nf(\sinh s)^{-2}(\cosh s)^2\\
     =& -(\partial_t+\Delta)\phi'+(\sinh s \cosh s)^{-1}(\partial_t+\Delta)\phi+\phi\gamma+\phi'\eta.
\end{align*}
By our notation, we can calculate that 
\begin{align*}
    \alpha=&2(1-n)+2u'-2(\sinh s )^{-2}\\
    =& 2(1-n)+2(1-n)(\sinh s)^{-2}+2(\cosh s)^{-2},\\
    \beta=& \Delta u+2(1-n)u+2(\sinh s)^{-3}(\cosh s)^{-1}\\
    =& -2(\cosh s)^{-3}\sinh s+(n-1)(\sinh s\cosh s)^{-1}\\
    &+(n-2)(3-n)(\sinh s)^{-3}\cosh s+2(1-n)\sinh s(\cosh s)^{-1}\\
    &+2(1-n)(n-2)(\sinh s)^{-1}\cosh s+2(\sinh s)^{-3}(\cosh s)^{-1},\\
\gamma=& \Delta(\sinh s\cosh s)^{-1}-2n(\sinh s)^{-3}\cosh s+2(\sinh s\cosh s)^{-1}\\
=& 3(1-n)(\sinh s)^{-3}\cosh s+2(\cosh s)^{-3}\sinh s+(3-n)(\sinh s\cosh s)^{-1},\\
\eta =& 2[(\sinh s\cosh s)^{-1}]'+2n(\sinh s)^{-2}(\cosh s)^2-2\\
=& 2(n-1)+2(n-1)(\sinh s)^{-2}-2(\cosh s)^{-2}.
\end{align*}

Thus, we can get
\begin{align*}
    A=&-B+[u+(\sinh s \cosh s)^{-1}](\partial_t+\Delta)\phi+\phi(\gamma+\beta)+\phi'(\alpha+\eta)\\
    =&-B+(n-1)(\sinh s )^{-1}\cosh s(\partial_t+\Delta)\phi\\
    &+\phi [2(\sinh s\cosh s)^{-1}+(-n^2+2n-3)(\sinh s)^{-3}\cosh s+2(1-n)\sinh s(\cosh s)^{-1}\\
    &\quad\quad+2(1-n)(n-2)(\sinh s)^{-1}\cosh s+2(\sinh s)^{-3}(\cosh s)^{-1}]\\
    =& -B+(n-1)(\sinh s )^{-1}\cosh s(\partial_t+\Delta)\phi\\
    &+\phi[-2(n-1)^2(\sinh s)^{-1}\cosh s+2(1-n)(\sinh s)^{-3}(\cosh s)^{-1}\\
     &\quad\quad+(n-3)(1-n)(\sinh s)^{-3}\cosh s]\\
     =&  -B+(n-1)(\sinh s )^{-1}\cosh s[\partial_t+\Delta+2(1-n)-2(\sinh s)^{-2}(\cosh s)^{-2}-(n-3)(\sinh s)^{-2}]\phi\\
     =&-B+(n-1)(\sinh s )^{-1}\cosh s[\partial_t+\Delta+2(1-n)+2(\cosh s)^{-2}-(n-1)(\sinh s)^{-2}]\phi.
\end{align*}
Since by assumption 
\begin{align*}
    (\partial_t+\Delta)\phi=[2(n-1)-2(\cosh s)^{-2}+(n-1)(\sinh s)^{-2}]\phi,
\end{align*}
and
\begin{align*}
    &(\partial_t+\Delta)\phi'=[(\partial_t+\Delta)\phi]'+(n-1)(\sinh s)^{-2}\phi',
\end{align*}
then we can calculate
\begin{align*}
    B&=-[(\partial_t+\Delta)\phi]'-(n-1)(\sinh s)^{-2}\phi'+(\sinh s \cosh s)^{-1}(\partial_t+\Delta)\phi+\phi\gamma+\phi'\eta\\
    &=\phi'E+\phi F,
\end{align*}
where 
\begin{align*}
    E=& -2(n-1)+2(\cosh s)^{-2}-(n-1)(\sinh s)^{-2}-(n-1)(\sinh s)^{-2}+\eta\\
    =& 0,\\
    F=& -4(\cosh s)^{-3}\sinh s+2(n-1)(\sinh s)^{-3}\cosh s+2(n-1)(\sinh s \cosh s)^{-1}\\
    &-2(\sinh s)^{-1}(\cosh s)^{-3}+(n-1)(\sinh s)^{-3}(\cosh s)^{-1}+\gamma\\
    =&0.
\end{align*}
Finally, we obtain 
\[ A=-B=0. \]

By Theorem \ref{thm3.7}, for $t\in [0, T_*]$, $|e(t)|=O(e^{-s\tau})$ and by Theorem \ref{thm5.2}, for $(x, t)\in \partial A(0, 0.9r, 1.1r)\times[0, \theta]$, we have 
\begin{align*}
    &\int_0^{\theta}\frac{d}{dt}Idt=I\bigg|_0^{\theta}=\int_{\partial A(0, 0.9r, 1.1r)}[\phi(s,\theta)e_{ij}(\theta)-\phi(s, 0)e_{ij}(0)](\nu^i\nu^j-b^{ij})d\mu_b,\\
    \leq & C(n, \varphi)\theta^{-\frac{n}{2}}\exp\left( -\frac{d_{ab}^2r^2}{c\theta} \right)e^{1.1(2n-\tau-1)r}.
\end{align*} 

Next we estimate $C$ and $D$.
By the expression of $Q$, we observe that
\begin{align*}
    &\int_{A(0, 0.9r, 1.1r)}[\phi'+u\phi]Qd\mu_b\\
    \leq &C(n) \int_{A(0, 0.9r, 1.1r)}(\phi'+u\phi)(|e|^2+|De|^2+D(e*De))d\mu_b\\
    \leq & C(n)\int_{A(0, 0.9r, 1.1r)}[|\phi'+u\phi|(|e|^2+|De|^2)+|\phi''+u\phi'+u'\phi||e||De|]d\mu_b\\
    &+C(n)\int_{\partial A(0, 0.9r, 1.1r)}|\phi'+u\phi||e||De|d\mu_b,
\end{align*}
and
\begin{align*}
   & \int_{\partial A(0, 0.9r, 1.1r)}[(\phi'+u\phi)b^{ij}D_{\nu}e_{ij}-D_{\nu}(\phi'+u\phi)tr_be]d\mu_b\\
    \leq& C(n)\int_{\partial A(0, 0.9r, 1.1r)}[|\phi'+u\phi||De|+|\phi''+u\phi'+u'\phi||e|]d\mu_b,
\end{align*}
which implies that 
\begin{align*}
    C
    \leq & C(n)\int_{A(0, 0.9r, 1.1r)}[|\phi'+u\phi|(|e|^2+|De|^2)+|\phi''+u\phi'+u'\phi||e||De|]d\mu_b\\
    &+C(n)\int_{\partial A(0, 0.9r, 1.1r)}[|\phi'+u\phi||e||De|+|\phi'+u\phi||De|+|\phi''+u\phi'+u'\phi||e|]d\mu_b.
\end{align*}
Similarly, we have
\begin{align*}
    D\leq& C(n)\int_{A(0, 0.9r, 1.1r)}[|f|(|e|^2+|De|^2)+|f'||e||De|] d\mu_b\\
    &+C(n)\int_{\partial A(0, 0.9r, 1.1r)}[|f||e||De|+|f||De|+|f'||e|] d\mu_b.
\end{align*}

Therefore, by the expression of $u$ and $f$, we get
\begin{align*}
    C+D\leq & C(n)\int_{A(0, 0.9r, 1.1r)}[(|\phi'|+|\phi|)(|e|^2+|De|^2)+(|\phi''|+|\phi'|+|\phi|e^{-2s})|e||De|]d\mu_b\\
    &+C(n)\int_{\partial A(0, 0.9r, 1.1r)}[(|\phi'|+|\phi|)|De|+(|\phi''|+|\phi'|+|\phi|e^{-2s})|e|]d\mu_b.
\end{align*}
By Theorem \ref{thm5.2}, we see that on $A(0, 0.9r, 1.1r)\times (0, \theta)$,
\begin{align*}
    |\phi'|+|\phi|&\leq C(n, \varphi)e^{s};\\
    |\phi''|+|\phi'|+|\phi e^{-2s}|&\leq C(n, \varphi)e^{s};
\end{align*}
and on $\partial A(0, 0.9r, 1.1r)\times (0, \theta)$,
\begin{align*}
    |\phi'|+|\phi|&\leq C(n, \varphi)\theta^{-\frac{n+1}{2}}\exp\left( -\frac{d_{ab}^2r^2}{c\theta} \right)e^{1.1nr};\\
    |\phi''|+|\phi'|+|\phi e^{-2s}|&\leq C(n, \varphi)\theta^{-\frac{n+2}{2}}\exp\left( -\frac{d_{ab}^2r^2}{c\theta} \right)e^{1.1nr}.
\end{align*}
Hence, we get
\begin{align*}
    C+D\leq &C(n, \varphi)\int_{A(0, 0.9r, 1.1r)}e^s(|e|^2+|De|^2+|e||De|)d\mu_b\\
    &+C(n, \varphi)\theta^{-\frac{n+2}{2}}\exp\left( -\frac{d_{ab}^2r^2}{c\theta} \right)e^{1.1nr}\int_{\partial A(0, 0.9r, 1.1r)}(|De|\theta^{\frac{1}{2}}+|e|)d\mu_b.
\end{align*}

Therefore, we obtain
\begin{align*}
    &\int_0^{\theta}\left| \frac{d}{dt}\left[M_{C^0}(g_t, \varphi_{\theta}(\cdot, t), r)r\int_{0.9}^{1.1}\varphi_{\theta}(l, t)dl\right] \right|dt\\
    \leq & C(n, \varphi)\theta^{-\frac{n}{2}}\exp\left( -\frac{d_{ab}^2r^2}{c\theta} \right)e^{1.1(2n-\tau-1)r}\\
    &+ C(n, \varphi)\int_0^{\theta}\int_{A(0, 0.9r, 1.1r)}e^s(|e|^2+|De|^2+|e||De|)d\mu_bdt\\
    &+C(n, \varphi)\theta^{-\frac{n+2}{2}}\exp\left( -\frac{d_{ab}^2r^2}{c\theta} \right)e^{1.1nr}\int_0^{\theta}\int_{\partial A(0, 0.9r, 1.1r)}(|De|\theta^{\frac{1}{2}}+|e|)d\mu_bdt.
\end{align*}

By Theorem \ref{thm3.7}, for $r>r_0$ sufficiently large, we have 
\[ |De(t)|\leq C(n)t^{-1/2}e^{-\tau s}, \]
and by Theorem \ref{thm3.8}, for any $x\in\mathbb{H}^n$, $0<r_1^2<T_*$, 
\begin{equation}\label{del2}
    ||De_t||_{L^2(B(x, r_1)\times(0, r_1^2))}\leq C(n)r_1^{\frac{n}{2}}e^{-\tau s(x)}.
\end{equation}
Then we can get
\begin{align*}
    &\int_0^{\theta}\int_{A(0, 0.9r, 1.1r)}e^s(|e|^2+|e||De|)d\mu_bdt\\
    \leq & C(n)e^{0.9(n-2\tau)r}\theta^{1/2},
\end{align*}
\begin{align*}
    \int_0^{\theta}\int_{\partial A(0, 0.9r, 1.1r)}(|De|\theta^{\frac{1}{2}}+|e|)d\mu_bdt
    \leq C(n)e^{k r}\theta,
\end{align*}
where $k=1.1(n-1-\tau)$ if $(n-1-\tau)>0$; $k=0.9(n-1-\tau)$ if $(n-1-\tau)<0$.
Let $R=\sqrt{\theta}$ and cover $A(0, 0.9r, 1.1r)$ by balls $\{B_b(x_\alpha,R)\}_\alpha$ with uniformly
bounded overlap. Since $R\le \sqrt{T_*}$,
hence
$$
e^{s(y)}e^{-2\tau s(x_\alpha)}
\le C e^{(1-2\tau)s(y)}.
$$
By \eqref{del2},
$$
\int_0^\theta\int_{B_b(x_\alpha,R)} |De|^2\,d\mu_b\,dt
\le
C C_0^2 e^{-2\tau s(x_\alpha)}R^n .
$$
Multiplying by $e^s$ and
summing over $\alpha$, we obtain
$$
\begin{aligned}
\int_0^\theta\int_{A(0, 0.9r, 1.1r)} e^s |De|^2\,d\mu_b\,dt
&\le
C C_0^2
\sum_\alpha e^{(1-2\tau)s(x_\alpha)}R^n \\
&\le
C C_0^2
\int_{A(0, 0.9r, 1.1r)} e^{(1-2\tau)s}\,d\mu_b\\
&\leq CC_0^2e^{0.9(n-2\tau)r} .
\end{aligned}
$$

Finally, we have
\begin{align*}
    &\int_0^{\theta}\left| \frac{d}{dt}\left[M_{C^0}(g_t, \varphi_{\theta}(\cdot, t), r)r\int_{0.9}^{1.1}\varphi_{\theta}(l, t)dl\right] \right|dt\\
    \leq & C(n, \varphi)e^{0.9(n-2\tau)r} ,
\end{align*}
for $\theta$ sufficiently small.
    
\end{proof}

\begin{thm}\label{thm5.6}
    Under the conditions of Theorem \ref{thm5.5},
we have, for $r$ sufficiently large and $0<\theta\le 1$,
$$
\int_0^\theta\left|\frac{d}{dt}
M_{C^0}(g_t,\varphi_\theta(\cdot,t),r)\right|\,dt
\le C r^{-1}e^{0.9(n-2\tau)r}+C r^{-1}e^{\mu_\tau r}\theta,
$$
where
$$
\mu_\tau:=
\max_{s\in[0.9r,1.1r]}\frac{(n-\tau)s}{r}=
\begin{cases}
1.1(n-\tau), & \tau\le n,\\
0.9(n-\tau), & \tau>n,
\end{cases}
$$
\end{thm}

\begin{proof}
 Set
$$
J(t):=\int_{0.9}^{1.1}\varphi_\theta(l,t)\,dl,
$$
$$
M(t):=M_{C^0}(g_t,\varphi_\theta(\cdot,t),r).
$$
and 
$$
N(t):=M(t)\,rJ(t).$$
Hence
$$
M'(t)=\frac{N'(t)}{rJ(t)}-\frac{N(t)}{rJ(t)}\frac{J'(t)}{J(t)}=\frac{N'(t)}{rJ(t)}-M(t)\frac{J'(t)}{J(t)}.
$$
Therefore,
$$
\int_0^\theta |M'(t)|\,dt\le\frac{1}{r\inf_{[0,\theta]}J}
\int_0^\theta |N'(t)|\,dt+\sup_{[0,\theta]}|M(t)|\int_0^\theta
\left|\frac{J'(t)}{J(t)}\right|\,dt .
$$
By Theorem 5.5,
$$
\int_0^\theta |N'(t)|\,dt\le
C e^{0.9(n-2\tau)r}.
$$
It remains to estimate $J'(t)$ and $J(t)$.  Let
$$
\phi(s,t)=\cosh s\,\psi\!\left(\frac{s}{r},t\right),
\qquad
l=\frac{s}{r},
\qquad
\psi(l,t)=\varphi_\theta(l,t).
$$
The equation for $\phi$ is 
$$
(\partial_t+\Delta)\phi
=
\left[
2n-2+(n-1)\operatorname{csch}^2s
-2\operatorname{sech}^2s
\right]\phi .
$$
So we obtain
\begin{align*}
\partial_t\psi
=&
-r^{-2}\psi_{ll}
-r^{-1}
\left[
2\tanh(rl)+(n-1)\coth(rl)
\right]\psi_l\\
&+
\left[
n-2+(n-1)\operatorname{csch}^2(rl)
-2\operatorname{sech}^2(rl)
\right]\psi .
\end{align*}
Consequently,
$$
\begin{aligned}
J'(t)
&=
\frac{d}{dt}\int_{0.9}^{1.1}\psi(l,t)\,dl \\
&=
\left[
-r^{-2}\psi_l
-r^{-1}
\left(
2\tanh(rl)+(n-1)\coth(rl)
\right)\psi
\right]_{0.9}^{1.1}
+
(n-2)\int_{0.9}^{1.1}\psi(l,t)\,dl,
\end{aligned}
$$
together with Theorem \ref{thm5.2}, we have
\begin{align*}
    |J'(t)-(n-2)J(t)|\leq Cr^{-1}\theta^{-(n+1)/2}
\exp\left(-\frac{d_{ab}^2r^2}{c\theta}\right)e^{1.1(n-1)r},
\end{align*}
since $J(\theta)>0$, Gronwall gives 
\[ J(t)\geq C(n, \varphi)>0, \]
for any $0\leq t\leq\theta$.
Thus
$$
\left|
\frac{J'(t)}{J(t)}
\right|
\le
C
+
C r^{-1}\theta^{-(n+1)/2}
\exp\!\left(-\frac{d_{ab}^2r^2}{c\theta}\right)
e^{1.1(n-1)r}.
$$
Integrating in time gives
$$
\int_0^\theta
\left|
\frac{J'(t)}{J(t)}
\right|\,dt
\le
C\theta
+
C r^{-1}\theta^{-(n-1)/2}
\exp\!\left(-\frac{d_{ab}^2r^2}{c\theta}\right)
e^{1.1(n-1)r}\leq C\theta.
$$
Finally, from Definition \ref{defc0}, using
$$
|e|\le C e^{-\tau s},
\qquad
\cosh s\le C e^s,
\qquad
d\mu_b\sim e^{(n-1)s}\,ds,
$$
on $A(0,0.9r,1.1r)$, we have
$$
|M(t)|
\le
C r^{-1}e^{\mu_\tau r}.
$$
Combining the estimates yields
$$
\int_0^\theta |M'(t)|\,dt
\le
C r^{-1}e^{0.9(n-2\tau) r}
+
C r^{-1}e^{\mu_\tau r}\theta,
$$
\end{proof}

\begin{thm}\label{thm5.7}
Let $(M^n, g)$ be a $C^0_{\tau}$-asymptotically hyperbolic manifold with fixed coordinate chart at infinity $\Phi$ and $\tau>\frac{119n}{128}$, $n\geq3$.
    Let $g_0$ be a continuous metric on $\mathbb{H}^n$ such that $g_0=\Phi_*g$ on $A(0, 0.8r, 12r)$ for $r>r_0$ sufficiently large and 
    \[ ||g_0-b||_{L^{\infty}(\mathbb{H}^n)}<\epsilon, \]
    for some $\epsilon<1$.
    Let $g_t=g(t)$ be the normalized $b$-flow as Theorem \ref{thm3.7}.
    Let 
    \begin{equation*}
        \mu_{\tau}=
        \begin{cases}
            1.1(n-\tau),  \quad\quad &\tau\leq n,\\
            0,&\tau\geq n.
        \end{cases}
    \end{equation*}
    Choose $r>r_0$  sufficiently large such that $e^{-r\eta}<T_*$, where $\eta\in (\mu_{\tau}, 0.09(2\tau-n))$ and $T_*$ being as in Theorem \ref{thm3.7}.
    
     Suppose $R_{C^0_{\beta}}(g)(x)\geq-n(n-1)$, for $\beta\in (0, \frac{1}{2})$ and for any $\Phi(x)\in A(0, 0.8r, 12r)$. 
     Let $\varphi, \bar{\varphi}$ be any two  smooth positive functions with support in $(0.9, 1.1)$ and with nonzero integrals over $(0.9, 1.1)$ and $\varphi_{\theta}(l, t)$, $\bar{\varphi}_{\theta}(l, t)$ be defined as \eqref{varphithe} for any $\theta<T_*$.
     
     Then we have, for all $r'\in [\frac{1.1}{0.9}r, 10r]$,  
    \[ M_{C^0}(g, \varphi_{e^{-r'\eta}}(\cdot, 0), r')-M_{C^0}(g, \varphi_{e^{-r\eta}}(\cdot, 0), r)\geq -C(n, \varphi, \tau, \beta, \eta)e^{-\kappa r}, \]
    and
   \[ M_{C^0}(g, \bar{\varphi}_{e^{-r'\eta}}(\cdot, 0), r')-M_{C^0}(g, \varphi_{e^{-r\eta}}(\cdot, 0), r)\geq -C(n, \varphi, \tau, \beta, \eta)e^{-\kappa r}, \]
where \[  \kappa<\min\{0.9(2\tau-n)-10\eta, \eta-\mu_{\tau}\}.\]
\end{thm}

\begin{proof}
Since $r'>r>1$ gives $e^{-r\eta}>e^{-r'\eta}$, we have
\begin{align*}
     &M_{C^0}(g, \varphi_{e^{-r'\eta}}(\cdot, 0), r')-M_{C^0}(g, \varphi_{e^{-r\eta}}(\cdot, 0), r)\\
     =& M_{C^0}(g, \varphi_{e^{-r'\eta}}(\cdot, 0), r')-M_{C^0}(g_{e^{-r'\eta}}, \varphi_{e^{-r'\eta}}(\cdot, e^{-r'\eta}), r')\\
     &+M_{C^0}(g_{e^{-r'\eta}}, \varphi_{e^{-r'\eta}}(\cdot, e^{-r'\eta}), r')-M_{C^0}(g_{e^{-r'\eta}},\varphi_{e^{-r\eta}}(\cdot,e^{-r'\eta}), r)\\
     &+M_{C^0}(g_{e^{-r'\eta}},\varphi_{e^{-r\eta}}(\cdot,e^{-r'\eta}), r)-M_{C^0}(g, \varphi_{e^{-r\eta}}(\cdot, 0), r)\\
     \geq& -\int_0^{{e^{-r'\eta}}}\bigg|\frac{d}{dt}M_{C^0}(g_t, \varphi_{e^{-r'\eta}}(\cdot, t), r')\bigg|dt\\
     & +M_{C^0}(g_{e^{-r'\eta}}, \varphi, r')-M_{C^0}(g_{e^{-r'\eta}},\varphi_{e^{-r\eta}}(\cdot,e^{-r'\eta}), r)\\
     & -\int_0^{{e^{-r'\eta}}}\bigg|\frac{d}{dt}M_{C^0}(g_t, \varphi_{e^{-r\eta}}(\cdot, t), r)\bigg|dt\\
     \geq& -C (r')^{-1}e^{-0.9(2\tau-n)r'}-C (r')^{-1}e^{\mu_\tau r'}e^{-r'\eta}\\
     &-C(e^{r'\eta}e^{0.9(n-2\tau)r}+e^{11nr}e^{-r'\eta\lambda})\\
     & -\int_0^{{e^{-r\eta}}}\bigg|\frac{d}{dt}M_{C^0}(g_t, \varphi_{e^{-r\eta}}(\cdot, t), r)\bigg|dt\\
     \geq& -C(n, \varphi, \tau, \beta, \eta)e^{-\kappa r},
\end{align*}
where, in the second inequality, we have used Theorem \ref{thm5.6} and Proposition \ref{prop51}, and 
\[  \kappa<\min\{0.9(2\tau-n)-10\eta, \eta-\mu_{\tau},\frac{1.1}{0.9}\lambda\eta-11n \},\]
where $\lambda>0$ is as in Proposition \ref{prop51}.
Similarly, we get
\begin{align*}
    &M_{C^0}(g, \bar{\varphi}_{e^{-r'\eta}}(\cdot, 0), r')-M_{C^0}(g, \varphi_{e^{-r\eta}}(\cdot, 0), r)\\
    =&M_{C^0}(g, \bar{\varphi}_{e^{-r'\eta}}(\cdot, 0), r')-M_{C^0}(g_{e^{-r'\eta}}, \bar{\varphi}_{e^{-r'\eta}}(\cdot, e^{-r'\eta}), r')\\
    &+M_{C^0}(g_{e^{-r'\eta}}, \bar{\varphi}_{e^{-r'\eta}}(\cdot, e^{-r'\eta}), r')-M_{C^0}(g_{e^{-r'\eta}},\varphi_{e^{-r\eta}}(\cdot,e^{-r'\eta}), r)\\
     &+M_{C^0}(g_{e^{-r'\eta}},\varphi_{e^{-r\eta}}(\cdot,e^{-r'\eta}), r)-M_{C^0}(g, \varphi_{e^{-r\eta}}(\cdot, 0), r)\\
     \geq& -\int_0^{{e^{-r'\eta}}}\bigg|\frac{d}{dt}M_{C^0}(g_t, \bar{\varphi}_{e^{-r'\eta}}(\cdot, t), r')\bigg|dt\\
     & +M_{C^0}(g_{e^{-r'\eta}}, \bar{\varphi}, r')-M_{C^0}(g_{e^{-r'\eta}},\varphi_{e^{-r\eta}}(\cdot,e^{-r'\eta}), r)\\
     & -\int_0^{{e^{-r'\eta}}}\bigg|\frac{d}{dt}M_{C^0}(g_t, \varphi_{e^{-r\eta}}(\cdot, t), r)\bigg|dt\\
     \geq& -C(n, \varphi, \tau, \beta, \eta)e^{-\kappa r}.
\end{align*}

\end{proof}

\begin{remark}
The proof of Corollary 5.4 in \cite{MR4685089} should be modified; otherwise, $\eta<2$ needs to be required.
\end{remark}

\begin{thm}
Let $(M^n, g)$ be a $C^0_{\tau}$-asymptotically hyperbolic manifold with the chart at infinity $\Phi$ and $\tau>\frac{119n}{128}$, $n\geq3$.
    Let $g_0$ be a continuous metric on $\mathbb{H}^n$  such that $g_0=\Phi_*g$ on $A(0, 0.8r, 12r)$ for $r>r_0$ sufficiently large and 
    \[ ||g_0-b||_{L^{\infty}(\mathbb{H}^n)}<\epsilon, \]
    for some $\epsilon<1$.
    Let $g_t=g(t)$ be the normalized $b$-flow as Theorem \ref{thm3.7}.
    Choose $\eta$ as in Theorem \ref{thm5.7}.
    
     Suppose $R_{C^0_{\beta}}(g)(x)\geq-n(n-1)$, for $\beta\in (0, \frac{1}{2})$ and for any $\Phi(x)\in A(0, 0.8r, 12r)$. 
     Let $\varphi$ be any   smooth positive function with  $supp\varphi\subset\subset(0.9, 1.1)$ and with nonzero integrals over $(0.9, 1.1)$ and $\varphi_{\theta}(l, t)$ be defined as \eqref{varphithe} for any $\theta<T'$.
     Then the limit 
     \[ \lim_{r\to\infty} M_{C^0}(g, \varphi_{e^{-r\eta}}(\cdot, 0), r)\]
exists, and either finite or $+\infty$. Furthermore, 
the limit $\lim_{r\to\infty} M_{C^0}(g, \varphi_{e^{-r\eta}}(\cdot, t), r)$ is finite if and only if the last condition in  Theorem 2.9 of \cite{MR4685089} is satisfied.
\end{thm}

\begin{proof}
    The proof is the same as Lemma 7.1 and Theorem 7.3 of \cite{MR4685089}.
\end{proof}

\textbf{Declarations:} 
\begin{enumerate}
    \item Conflict of interest: We declare that we have no conflict of interest.
\item Data availability: We declare that the data supporting the findings of this study are available within the paper.
\end{enumerate}

\hspace{0.4cm}

\bibliography{ref}

@misc{gicquaud2025,
      title={A definition of the mass aspect function for weakly regular asymptotically hyperbolic manifolds}, 
      author={Romain Gicquaud and Anna Sakovich},
      year={2025},
      eprint={2502.00125},
      archivePrefix={arXiv},
      primaryClass={math.DG},
      url={https://arxiv.org/abs/2502.00125}, 
}

@misc{lundgren2025,
      title={A generalization of the ADM mass for asymptotically Euclidean manifolds of weak regularity}, 
      author={Stig Lundgren and Benjamin Meco},
      year={2025},
      eprint={2506.02670},
      archivePrefix={arXiv},
      primaryClass={math.DG},
      url={https://arxiv.org/abs/2506.02670}, 
}

@article {MR4685089,
    AUTHOR = {Burkhardt-Guim, Paula},
     TITLE = {A{DM} mass for {$C^0$} metrics and distortion under
              {R}icci-{D}e{T}urck flow},
   JOURNAL = {J. Reine Angew. Math.},
  FJOURNAL = {Journal f\"{u}r die Reine und Angewandte Mathematik. [Crelle's
              Journal]},
    VOLUME = {806},
      YEAR = {2024},
     PAGES = {187--245},
      ISSN = {0075-4102,1435-5345},
   MRCLASS = {53E20 (53C20)},
  MRNUMBER = {4685089},
       DOI = {10.1515/crelle-2023-0085},
       URL = {https://doi.org/10.1515/crelle-2023-0085},
}

@article {LeePh,
    AUTHOR = {Lee, Dan A. and LeFloch, Philippe G.},
     TITLE = {The positive mass theorem for manifolds with distributional
              curvature},
   JOURNAL = {Comm. Math. Phys.},
  FJOURNAL = {Communications in Mathematical Physics},
    VOLUME = {339},
      YEAR = {2015},
    NUMBER = {1},
     PAGES = {99--120},
      ISSN = {0010-3616,1432-0916},
   MRCLASS = {53C20 (53C21)},
  MRNUMBER = {3366052},
MRREVIEWER = {Hans-Peter\ K\"unzle},
       DOI = {10.1007/s00220-015-2414-9},
       URL = {https://doi.org/10.1007/s00220-015-2414-9},
}

@article {MSimon,
    AUTHOR = {Simon, Miles},
     TITLE = {Deformation of {$C^0$} {R}iemannian metrics in the direction
              of their {R}icci curvature},
   JOURNAL = {Comm. Anal. Geom.},
  FJOURNAL = {Communications in Analysis and Geometry},
    VOLUME = {10},
      YEAR = {2002},
    NUMBER = {5},
     PAGES = {1033--1074},
      ISSN = {1019-8385,1944-9992},
   MRCLASS = {53C44},
  MRNUMBER = {1957662},
MRREVIEWER = {Peng\ Lu},
       DOI = {10.4310/CAG.2002.v10.n5.a7},
       URL = {https://doi.org/10.4310/CAG.2002.v10.n5.a7},
}

@article {paula19,
    AUTHOR = {Burkhardt-Guim, Paula},
     TITLE = {Pointwise lower scalar curvature bounds for {$C^0$} metrics
              via regularizing {R}icci flow},
   JOURNAL = {Geom. Funct. Anal.},
  FJOURNAL = {Geometric and Functional Analysis},
    VOLUME = {29},
      YEAR = {2019},
    NUMBER = {6},
     PAGES = {1703--1772},
      ISSN = {1016-443X,1420-8970},
   MRCLASS = {53E20 (53C21)},
  MRNUMBER = {4034918},
MRREVIEWER = {Kin\ Ming\ Hui},
       DOI = {10.1007/s00039-019-00514-3},
       URL = {https://doi.org/10.1007/s00039-019-00514-3},
}

@article {jang21,
    AUTHOR = {Sakovich, Anna},
     TITLE = {The {J}ang equation and the positive mass theorem in the
              asymptotically hyperbolic setting},
   JOURNAL = {Comm. Math. Phys.},
  FJOURNAL = {Communications in Mathematical Physics},
    VOLUME = {386},
      YEAR = {2021},
    NUMBER = {2},
     PAGES = {903--973},
      ISSN = {0010-3616,1432-0916},
   MRCLASS = {53C20},
  MRNUMBER = {4294283},
MRREVIEWER = {Xiao\ Zhang},
       DOI = {10.1007/s00220-021-04083-1},
       URL = {https://doi.org/10.1007/s00220-021-04083-1},
}

@article {Wang01,
    AUTHOR = {Wang, Xiaodong},
     TITLE = {The mass of asymptotically hyperbolic manifolds},
   JOURNAL = {J. Differential Geom.},
  FJOURNAL = {Journal of Differential Geometry},
    VOLUME = {57},
      YEAR = {2001},
    NUMBER = {2},
     PAGES = {273--299},
      ISSN = {0022-040X,1945-743X},
   MRCLASS = {53C20 (53C21)},
  MRNUMBER = {1879228},
MRREVIEWER = {Lars\ \AA ke\ Andersson},
       URL = {http://projecteuclid.org/euclid.jdg/1090348112},
}

@article {Chrusciel03,
    AUTHOR = {Chru\'sciel, Piotr T. and Herzlich, Marc},
     TITLE = {The mass of asymptotically hyperbolic {R}iemannian manifolds},
   JOURNAL = {Pacific J. Math.},
  FJOURNAL = {Pacific Journal of Mathematics},
    VOLUME = {212},
      YEAR = {2003},
    NUMBER = {2},
     PAGES = {231--264},
      ISSN = {0030-8730,1945-5844},
   MRCLASS = {53C21 (83C05 83C30)},
  MRNUMBER = {2038048},
MRREVIEWER = {Simonetta\ Frittelli},
       DOI = {10.2140/pjm.2003.212.231},
       URL = {https://doi.org/10.2140/pjm.2003.212.231},
}

@article {wangzh19,
    AUTHOR = {Wang, Yaohua and Zhang, Xiao},
     TITLE = {Positive energy theorem for asymptotically anti--de {S}itter
              spacetimes with distributional curvature},
   JOURNAL = {Internat. J. Math.},
  FJOURNAL = {International Journal of Mathematics},
    VOLUME = {30},
      YEAR = {2019},
    NUMBER = {13},
     PAGES = {1940003, 16},
      ISSN = {0129-167X,1793-6519},
   MRCLASS = {53C27 (53C80 83C40)},
  MRNUMBER = {4043999},
MRREVIEWER = {Volker\ Branding},
       DOI = {10.1142/S0129167X19400032},
       URL = {https://doi.org/10.1142/S0129167X19400032},
}

@article {RFofah12,
    AUTHOR = {Balehowsky, T. and Woolgar, E.},
     TITLE = {The {R}icci flow of asymptotically hyperbolic mass and
              applications},
   JOURNAL = {J. Math. Phys.},
  FJOURNAL = {Journal of Mathematical Physics},
    VOLUME = {53},
      YEAR = {2012},
    NUMBER = {7},
     PAGES = {072501, 15},
      ISSN = {0022-2488,1089-7658},
   MRCLASS = {53C44},
  MRNUMBER = {2985224},
MRREVIEWER = {Shu-Yu\ Hsu},
       DOI = {10.1063/1.4732118},
       URL = {https://doi.org/10.1063/1.4732118},
}

@article {Bahu11,
    AUTHOR = {Bahuaud, Eric},
     TITLE = {Ricci flow of conformally compact metrics},
   JOURNAL = {Ann. Inst. H. Poincar\'e{} C Anal. Non Lin\'eaire},
  FJOURNAL = {Annales de l'Institut Henri Poincar\'e{} C. Analyse Non
              Lin\'eaire},
    VOLUME = {28},
      YEAR = {2011},
    NUMBER = {6},
     PAGES = {813--835},
      ISSN = {0294-1449,1873-1430},
   MRCLASS = {53C44 (35R01 53C21)},
  MRNUMBER = {2859929},
MRREVIEWER = {Meng\ Zhu},
       DOI = {10.1016/j.anihpc.2011.03.007},
       URL = {https://doi.org/10.1016/j.anihpc.2011.03.007},
}

@article {QingShiWu13,
    AUTHOR = {Qing, Jie and Shi, Yuguang and Wu, Jie},
     TITLE = {Normalized {R}icci flows and conformally compact {E}instein
              metrics},
   JOURNAL = {Calc. Var. Partial Differential Equations},
  FJOURNAL = {Calculus of Variations and Partial Differential Equations},
    VOLUME = {46},
      YEAR = {2013},
    NUMBER = {1-2},
     PAGES = {183--211},
      ISSN = {0944-2669,1432-0835},
   MRCLASS = {53C25 (58J05)},
  MRNUMBER = {3016507},
MRREVIEWER = {L\'eonard\ Todjihounde},
       DOI = {10.1007/s00526-011-0479-7},
       URL = {https://doi.org/10.1007/s00526-011-0479-7},
}

@book {topping,
    AUTHOR = {Topping, Peter},
     TITLE = {Lectures on the {R}icci flow},
    SERIES = {London Mathematical Society Lecture Note Series},
    VOLUME = {325},
 PUBLISHER = {Cambridge University Press, Cambridge},
      YEAR = {2006},
     PAGES = {x+113},
      ISBN = {978-0-521-68947-2; 0-521-68947-3},
   MRCLASS = {53C44},
  MRNUMBER = {2265040},
MRREVIEWER = {Peng\ Lu},
       DOI = {10.1017/CBO9780511721465},
       URL = {https://doi.org/10.1017/CBO9780511721465},
}

@article {Shiwx,
    AUTHOR = {Shi, Wan-Xiong},
     TITLE = {Deforming the metric on complete {R}iemannian manifolds},
   JOURNAL = {J. Differential Geom.},
  FJOURNAL = {Journal of Differential Geometry},
    VOLUME = {30},
      YEAR = {1989},
    NUMBER = {1},
     PAGES = {223--301},
      ISSN = {0022-040X,1945-743X},
   MRCLASS = {58G30 (53C20 58D25 58G11)},
  MRNUMBER = {1001277},
MRREVIEWER = {Friedbert\ Pr\"ufer},
       URL = {http://projecteuclid.org/euclid.jdg/1214443292},
}

@book {Chow,
    AUTHOR = {Chow, Bennett and Chu, Sun-Chin and Glickenstein, David and
              Guenther, Christine and Isenberg, James and Ivey, Tom and
              Knopf, Dan and Lu, Peng and Luo, Feng and Ni, Lei},
     TITLE = {The {R}icci flow: techniques and applications. {P}art {III}.
              {G}eometric-analytic aspects},
    SERIES = {Mathematical Surveys and Monographs},
    VOLUME = {163},
 PUBLISHER = {American Mathematical Society, Providence, RI},
      YEAR = {2010},
     PAGES = {xx+517},
      ISBN = {978-0-8218-4661-2},
   MRCLASS = {53C44 (35K08 35K55)},
  MRNUMBER = {2604955},
MRREVIEWER = {James Alexander McCoy},
       DOI = {10.1090/surv/163},
       URL = {https://doi.org/10.1090/surv/163},
}

@article {Schnurer,
    AUTHOR = {Schn\"urer, Oliver C. and Schulze, Felix and Simon, Miles},
     TITLE = {Stability of hyperbolic space under {R}icci flow},
   JOURNAL = {Comm. Anal. Geom.},
  FJOURNAL = {Communications in Analysis and Geometry},
    VOLUME = {19},
      YEAR = {2011},
    NUMBER = {5},
     PAGES = {1023--1047},
      ISSN = {1019-8385,1944-9992},
   MRCLASS = {53C44 (35K55)},
  MRNUMBER = {2886716},
MRREVIEWER = {Qihua\ Ruan},
       DOI = {10.4310/CAG.2011.v19.n5.a8},
       URL = {https://doi.org/10.4310/CAG.2011.v19.n5.a8},
}

\end{document}